%% file: Final_version_sent_to_journal.tex
\documentclass[11pt]{article}

\input{preamble}
\title{Inverting Laguerre tessellations: Recovering 
tessellations 
from the volumes and centroids of their cells
using optimal transport}
\author{D.~P.~Bourne\footnote{Maxwell Institute for Mathematical Sciences and Department of Mathematics, Heriot-Watt University, Edinburgh, UK}, $\;$ M.~Pearce\footnotemark[\value{footnote}], $\;$ S.~M.~Roper\footnote{School of Mathematics and Statistics, University of Glasgow, Glasgow, UK}}
\date{\today}

\begin{document}
    
\maketitle

\begin{abstract}
In this paper we study an inverse problem in convex geometry, inspired by a problem in materials science. Firstly, we consider the question of whether a Laguerre tessellation (a partition by convex polytopes) can be recovered from only the volumes and centroids of its cells. We show that this problem has a unique solution and give a constructive way of computing it using optimal transport theory and convex optimisation. Secondly, we consider the problem of fitting a Laguerre tessellation to synthetic volume and centroid data. Given some target volumes and centroids, we seek a Laguerre tessellation such that the difference between the volumes and centroids of its cells and the target volumes and centroids is minimised. For an appropriate objective function and suitable data, we prove that local minimisers of this problem can be constructed using convex optimisation. We also illustrate our results numerically.
There is great interest in the computational materials science community in fitting Laguerre tessellations to electron backscatter diffraction (EBSD) and x-ray diffraction images of polycrystalline materials. 
As an application of our results we fit a 2D Laguerre tessellation to an EBSD image of steel.
\end{abstract}

\section{Introduction}

In this paper we study an inverse problem in computational geometry using tools from optimal transport theory.
Laguerre tessellations, also known as power diagrams, are a 
generalisation of Voronoi tessellations \cite{AurenhammerKleinLee,BWY06,Okabe2000}. A Laguerre tessellation is a partition of a set into convex polytopes; see equation \eqref{eq:LaguerreTesselation} for the definition. Our goals are the following: 
\begin{itemize}
    \item Goal 1 (inverse problem): Given only the volumes and centroids (barycenters) of the cells in a Laguerre tessellation, can you recover the tessellation?  
    \item Goal 2 (fitting problem): Given a list of target volumes and centroids, can you find a Laguerre tessellation 
    such that the volumes and centroids of the Laguerre cells give the `best fit' of the target volumes and centroids?
\end{itemize}
These goals are stated precisely in Section \ref{sec:statement of problem}.

\paragraph{Motivation.}
There are numerous applications of Laguerre tessellations including cell biology \cite{BockEtAl2010,DiezFeydy2024}, domain decomposition \cite{LevyCentroidalPower}, fluid mechanics \cite{Gallouet:2022}, foam modelling \cite{BusaryevEtAl2012,Liebscher2015}, image interpolation \cite{LevySemiDiscrete2015}, mesh generation \cite{Levy2022}, optics \cite{MeyronMerigotThibert2018}, optimal location problems \cite{BourneRoper2015}, sampling \cite{MerigotSantambrogioSarrazin} and the proof of the Nandakumar and Ramana Rao conjecture \cite{BlagojevicZiegler}.

Our main motivation comes from materials science. 
There is a vast engineering literature on using 
Laguerre tessellations 
to model the microstructure of foams and metals, including \cite{FWZL04,KuhnSteinhauser08,WuZhouWangYang2010,LLLFP11,SBDWKKS2016,DepriesterKubler2019,KSSB20,BPR23}. 
Laguerre tessellations are even used in the steel industry, by Tata Steel Reseach and Development \cite{BKRS20}, and are part of standard software libraries for synthetic microstructure generation such as DREAM.3D \cite{groeber2014dream} and Neper \cite{QueyEtAl2011}.
Laguerre tessellations are also used in image processing, to compress or reconstruct 
electron backscatter diffraction (EBSD) or x-ray diffraction (XRD) images of metals \cite{SBDWKKS2016,PetrichEtAl2019} (see below for more references).
For these applications it is important to have efficient algorithms for generating Laguerre tessellations with prescribed geometrical and statistical properties, such as the volumes of the cells, their shape and their spatial distribution. In this paper we focus on controlling the volumes and the centroids of the cells, which give us control over their sizes and locations. This forms part of an ongoing research programme initiated by the authors in \cite{BKRS20} of leveraging optimal transport theory for microstructure modelling. We also believe that our results will be of interest in their own right in the computational geometry community (especially the results for Goal 
1), where there is a large literature on Laguerre diagrams including \cite{Aurenhammer87,Aurenhammer98,AurenhammerKleinLee,BWY06,BriedenGritzmann2012,Okabe2000}.

\paragraph{Previous work.}
The phrase \emph{inverting Laguerre tessellations} comes from the paper \cite{DuanEtAl}, where the authors studied the problem of recovering the \emph{generators} of a 2D normal Laguerre tessellation from the tessellation itself (from the edges and vertices). The generators are not uniquely determined by the tessellation, but they can be recovered uniquely once you have fixed the generator (seed and weight) of one interior cell and one coordinate of the seed of an adjacent cell; see \cite[Theorem 3.1]{DuanEtAl}. This non-uniqueness is also discussed in Lemma \ref{Lemma:Meyron} below. In \cite{DuanEtAl} the authors solved the same problem numerically using the cross-entropy method. 

The inverse problem that we study in Goal 1 is a little different; we seek to recover the tessellation given the volumes and centroids of the cells. In the process we recover one set of generators from the infinite family of possible generators.
Another inverse problem is to determine whether a partition by convex polytopes is a Laguerre tessellation. This is addressed in \cite{Aurenhammer87} and \cite[Theorem 3.2]{Lautensack_Zuyev_2008}. 
A discrete version of Goal 1 is studied in \cite[Theorem 2.1]{BriedenGritzmann2012}, where the authors give necessary and sufficient conditions for a discretised version of a Laguerre diagram to have cells of given (discrete) volumes and centroids.

Several variations of Goal 2 have been studied in the materials science literature, motivated by the application of 
fitting Laguerre diagrams and anisotropic Laguerre diagrams to images of metals, such as 
EBSD or XRD
images \cite{ABGLP15,SBDWKKS2016,VBWBSJ16,QR18,PetrichEtAl2019,PFWKS21,AFGK23}.
In Section \ref{sec:EBSD} we fit a Laguerre tessellation to a 2D EBSD image of steel. 
An EBSD image can be thought of 
as an assignment map, which associates each pixel in the image to a `cell' (grain) in the metal. 

The main difference between our work and these previous works is the choice of objective function. In our work the objective function measures the mismatch between the actual centroids of the Laguerre cells and the target centroids, which are the centroids of the cells in the EBSD image. The areas of the cells are fitted `exactly', up to any desired tolerance. In most of the previous works listed above, the objective function measures the number of misassigned pixels in the image (the pixels assigned to the wrong Laguerre cells). 
The objective functions in \cite{PetrichEtAl2019} and \cite[Section 4]{QR18} are closest in spirit to ours, where the authors minimise a weighted sum of the volume and centroid error.
Moreover, most of the works above use the full EBSD/XRD image (assignment map), whereas in our work and \cite{PetrichEtAl2019,QR18} only the areas and centroids of the cells in the image are used, and the rest of the data is discarded. While we do not exploit the full dataset, one advantage of our approach is that it can still be used when the EBSD/XRD image is incomplete or corrupted. Also, our approach can also be used to generate synthetic microstructures in the absence of EBSD/XRD data, which is time-consuming to obtain, especially in 3D.

Another novelty of our work is the choice of optimisation algorithm. We show how Goals 1 and 2 can be reformulated as \emph{convex} problems. 
A convex formulation of the fitting problem is also given in \cite[equation (DiLPM)]{AFGK23}, where they formulate it as linear programmming problem. However, it is very high-dimensional and hence computationally expensive; the number of unknowns is of the order of the number of pixels/voxels in the image, whereas in our case the number of unknowns is of the order of the number of Laguerre cells. Note, however, that the authors of \cite{AFGK23} also propose a high-accuracy, fast heuristic for approximating the solution of the linear programming problem. As far as we are aware, all the other approaches to the fitting problem in the engineering literature use slower, \emph{non-convex} optimisation methods, e.g., the stochastic cross-entropy method is used in \cite{PFWKS21,PetrichEtAl2019} and a gradient-free method (a generalisation of the Nelder–Mead method) is used in \cite{QR18}.

An entirely different approach is given in \cite{TefraRowenhurtsDirectParameterEstimateforGPBD}, where they propose an optimisation-free method for fitting an (anisotropic) Laguerre tessellation to an EBSD image. Like \cite{AFGK23}, their heuristic gives good results for little computational effort. In this paper we give a partial explanation for this, as a by-product of our solution of Goal 2.

\paragraph{Contributions.} Here is a summary of our main contributions:
\begin{itemize}
    \item Goal 1: We prove that a Laguerre tessellation is uniquely determined by the volumes and centroids of its cells (Theorem \ref{thm:uniqueness}) and that it can be recovered by solving an unconstrained convex optimisation problem (Theorem \ref{cor:soln to Goal 1}). In the language of machine learning, one would say that a Laguerre tessellation can be learnt from the volumes and centroids of its cells.
    We generalise these results to anisotropic Laguerre tessellations (Section \ref{sec:aniso}). 
    \item Goal 2: We prove that, under suitable assumptions on the target volumes and centroids, fitting a Laguerre tessellation to volume and centroid data is also essentially a convex optimisation problem, in the sense that local minimisers of the least-squares fitting error can be found by solving a constrained convex optimisation problem (Theorem \ref{thm:main}). The assumptions on the target volumes and centroids, however, are implicit and potentially quite restrictive, as illustrated numerically in Section \ref{sec:fitting} 
    \item Justification of a popular heuristic in the literature: A popular heuristic for approximately solving Goal 2 is to take the seeds of the Laguerre diagram to be the target centroids \cite{TefraRowenhurtsDirectParameterEstimateforGPBD}. In Remark \ref{remark:heuristic} we give a partial justification for why this heuristic works so well. 
    \item Application in computational materials science: As an application of Goal 2, we fit a Laguerre tessellation to an EBSD image of a single-phase steel, provided to us by Tata Steel Europe (Section \ref{sec:EBSD}).
    \item Python code for the numerical experiments in Section \ref{sec: numerics} can be found on GitHub: \url{https://github.com/DPBourne/Laguerre-Polycrystalline-Microstructures}
\end{itemize}

\paragraph{Methods.} Our main theoretical tool is \emph{optimal transport theory} \cite{SantambrogioBook}, in particular \emph{semi-discrete optimal transport} \cite[Section 4]{MerigotThibertOT}, which has strong links to computational geometry. In Section \ref{subsec:OT} we recall how a Laguerre tessellation with cells of prescribed volumes can be computed efficiently by solving a semi-discrete optimal transport problem; see Theorem \ref{Thm:DualOT}. The simulations in Section \ref{sec: numerics} rely on a modern numerical method for optimal transport, the damped Newton method \cite{Kitagawa_Merigot_Thibert_2019}, and a fast Python implementation using the library \emph{pysdot} \cite{pysdot}.

\paragraph{Outline of the paper.} Section \ref{sec: Prelim} includes the necessary background material on Laguerre tessellations and semi-discrete optimal transport theory. The goals of the paper are stated in Section \ref{sec:statement of problem}. Section \ref{sec:main results} contains our main theoretical results. Section \ref{subsec:H} introduces the concave function $H$, which is important throughout the rest of the paper. Goal 1 is addressed in Section \ref{sec:uniqueness} and Goal 2 is addressed in Sections \ref{sec:existence} and \ref{sec:rewriting NLS}. Section \ref{sec: numerics} contains our numerical results. In particular, we solve Goal 1 numerically in Section \ref{sec: numerics - diagram exists} and Goal 2 numerically in Section \ref{sec:fitting}. In Section \ref{sec:EBSD} we give an application of Goal 2 in computational materials science (fitting a Laguerre tessellation to an EBSD image).
Finally, in Section \ref{sec:aniso} we show how our results for Goal 1 can be generalised to anisotropic Laguerre diagrams.

\section{Background material}
\label{sec: Prelim}     
    \subsection{Laguerre tessellations}
        Let $d \in \mathbb{N}$, $d \ge 2$. Throughout this paper we take $\Omega \subset \RR^d$ to be a non-empty, convex, bounded set that is equal to the closure of its interior. In the numerical experiments in Section \ref{sec: numerics} we take $d=2$ and $\Omega$ to be a rectangular domain. A Laguerre tessellation of $\Omega$ is a collection of sets that partitions $\Omega$ into convex regions {(convex polytopes if $\Omega$ is a polytope)}. The partition is parameterised by a set of seeds and a corresponding set of weights. Define the set of distinct seeds by
    $$\distinct_{n} = \{\bfX = (x_1,...,x_n) \in (\RR^{d})^n: x_i \neq x_j \; \forall \; i,j \in \{1,\ldots,n\}, \, i \neq j\}.$$
    Given $\bfX\in\distinct_n$ and a set of weights $\bfw = (w_1,\ldots,w_n)$, $w_i \in \RR$, the $i^{\text{th}}$ \emph{Laguerre cell} generated by $(\bfX,\bfw)$ is defined by 
        \begin{equation}
    \label{eq:LaguerreTesselation}
            \Lag_i(\bfX,\mathbf{w}) = \{x \in \Omega: |x-x_i|^2 - w_i \leq |x-x_j|^2 - w_j\; \forall\, j \in \{1,\ldots,n\}\}, 
        \end{equation}
        where $|\cdot|$ denotes the standard Euclidean norm. The \emph{Laguerre tessellation} generated by $(\bfX,\bfw)$ is the collection of cells $\{\Lag_i(\bfX,\bfw)\}_{i=1}^n$. To see that the Laguerre cells are convex, note that {they can be written as an intersection of convex sets:} 
        \begin{equation}
             \label{eq:LThalfplane}
             \Lag_i(\bfX,\bfw) = 
             \Omega \cap
        \bigcap_{j \ne i} H_{ij}(\bfX,\bfw),
        \end{equation}
        where $H_{ij}$ is the half-space
\[
H_{ij}(\bfX,\bfw) = \{x \in {\mathbb{R}^d}:|x-x_i|^2 - w_i \leq |x-x_j|^2 - w_j\}.
\]
Furthermore, all the cells are convex polygons (if $d=2$) or convex polyhedra (if $d=3$), except possibly for the cells that intersect the boundary of $\Omega$.

    Definition \eqref{eq:LaguerreTesselation} can be extended to the case where the seeds are non-distinct, however the corresponding cells do not necessarily form a tessellation. If two seeds coincide their cells are either equal or one is the empty-set, depending on the values of the corresponding weights.
    
         The parameterisation of the tessellation is non-unique; observe that if we add a constant to all the weights {($w_i \mapsto w_i + c$ for all $i$)}, then the inequalities in \eqref{eq:LaguerreTesselation} are still satisfied and the cells remain unchanged. 
         Moreover, the following lemma 
         asserts that we can uniformly translate and dilate the seeds without changing the Laguerre tessellation, provided that the weights are adjusted appropriately.
        \begin{lemma}[See \cite{Meyron2019},  Proposition 6] 
        \label{Lemma:Meyron}
            Let $\bfX_1 = (x^1_1,\ldots,x^1_n) \in \distinct_{n}$ and $\bfw_1 = (w_1^1,\ldots,w_n^1) \in \RR^{n}$ be the generators of a Laguerre tessellation. Given a dilation factor $\lambda > 0$ and translation $t \in \RR^d$, define $\bfX_2 = (x^2_1,\ldots,x^2_n)$ by $x^2_i = \lambda x^1_i + t$. Let
            \begin{equation} \label{eq:Lemma2.1Wformula}
            w^2_i = \lambda{w^1_i}
            + 2 \lambda t \cdot x_i^1  +\lambda(\lambda-1)| x_i^1|^2, \quad i \in \{1,\ldots,n\}. 
            \end{equation}
            Then
            $\Lag_i(\bfX_1,\bfw_1) = \Lag_i(\bfX_2,\bfw_2)$ for all $i \in \{1,\ldots,n\}$.  
        \end{lemma}
        
        Given a {(Lebesgue-measurable)} set $A \subseteq \mathbb{R}^d$, we define $\vol(A)$ to be its {$d$-dimensional volume (area if $d=2$, volume if $d=3$) and $\setcent(A)$ to be its centroid:}
        \begin{align*}
            \vol(A) = \int_{A} 1\, \text{d}x, \qquad \setcent(A) = \frac{1}{\vol(A)} \int_{A} x \, \text{d}x.
        \end{align*}
        The latter is only defined for $A$ such that $\vol(A)$ is non-zero.   
        Given $(\bfX,\bfw) \in \distinct_n \times \mathbb{R}^n$, we denote the volume and centroid of the Laguerre cell $\Lag_i(\bfX,\bfw)$  by 
    \begin{equation}
    \label{eq:m xi}
            \cellvol_i(\bfX,\bfw) = \vol(\Lag_i(\bfX,\bfw)),
            \qquad \cellcent_i(\bfX,\bfw) = \setcent(\Lag_i(\bfX,\bfw)). 
        \end{equation}
       Since the Laguerre cells are disjoint (up to sets of measure zero), we have
       \begin{align}
        \label{eq:MassOmega}
            \vol(\Omega) &= \int_{\Omega} 1\, \text{d}x = \sum_{i=1}^{n} \int_{\Lag_i(\bfX,\bfw)} 1\, \text{d}x = \sum_{i=1}^{n} \cellvol_i(\bfX,\bfw),  \\
        \label{eq:MassCentsOmega}
            \vol(\Omega) \cdot \setcent(\Omega) &= \int_{\Omega} x \, \text{d}x = \sum_{i=1}^{n} \int_{\Lag_i(\bfX,\bfw)}x \, \text{d}x = 
\sum_{i=1}^n\cellvol_i(\bfX,\bfw)\cellcent_i(\bfX,\bfw).
        \end{align}
      
        There is a one-to-one correspondence between Laguerre tessellations and solutions of semi-discrete optimal transport problems, which 
        is useful for generating Laguerre tessellations with cells of prescribed volumes, as described in the next section.
        
    \subsection{Link to semi-discrete optimal transport} \label{subsec:OT}
        Define $\RR_+=(0,\infty)$ and 
        \[
        \mathbb{V}_n = \left\{\bfv = (v_1,\ldots,v_n) \in \RR_{+}^{n} \,: \, \sum_{i=1}^{n}v_i = \vol(\Omega)\right\}.
        \]
        Given $\bfX = (x_1,\ldots,x_n) \in (\RR^{d})^n$ and $v \in \mathbb{V}_n$
        , we define the  measure $\nu(\bfX,{\bfv}) \in \mathcal{M}_{+}(\mathbb{R}^d)$ by
        \[
        \nu(\bfX,{\bfv})= \sum_{i=1}^{n} \tvoli_i \delta_{x_i},
        \]
        where $\delta_{x_i}$ denotes the Dirac measure supported at $\{x_i\}$. This is well defined even if the seeds $x_i$ are not distinct.
        
        Let $\ldomega$ be the $d$-dimensional Lebesgue measure restricted to $\Omega$. Any map $T:\Omega \to \{x_1,\ldots,x_n\}$
        satisfying $\vol(T^{-1}(\{x_i\})) = \tvoli_i$ for all $i$ is called an \emph{admissible transport map}. We denote the set of admissible transport maps by $\mathcal{A}(\bfX,\bfv)$:
        \begin{equation}
        \label{eq:admissible maps}
        \mathcal{A}(\bfX,\bfv)=\{T:\Omega 
        \to \{ x_1,\ldots,x_n\}
        \; : \; \vol(T^{-1}(\{x_i\})) = \tvoli_i \; \forall \; i \in \{1,\ldots,n\}\}.
        \end{equation}
For any $T\in\mathcal{A}(\bfX,\bfv)$ we define the \emph{transport cost} $M(T)$ by
\[
M(T):= \int_{\Omega} |x-T(x)|^2\,\mathrm{d}x = \sum_{i=1}^n \int_{T^{-1}(\{ x_i \})} |x-x_i|^2 \, \mathrm{d}x.
\]
The \emph{optimal} cost of transporting $\ldomega$ to $\nu(\bfX,\bfv)$ is 
        \begin{equation}
        \label{def:OTproblem} 
            %\mathcal{T}(\bfX,\bfw)
            \mathcal{T}(\bfX,\bfv)
            =\inf_{T\in\mathcal{A}(\bfX,\bfv)} M(T).
       \end{equation}
       Any minimiser $T^*$ in \eqref{def:OTproblem} is called an \emph{optimal transport map}, and 
       \begin{equation}
\label{eq:W2}
W_2(\mathcal{L}^d_\Omega,\nu(\bfX,\bfv)):=\mathcal{T}(\bfX,\bfv)^{1/2}
\end{equation}
is the \emph{Wasserstein distance} between the measures $\mathcal{L}^d_\Omega$ and  $\nu(\bfX,\bfw)$.  See for example \cite{SantambrogioBook}

       It is well-known, see for instance \cite{MerigotThibertOT},  that there exists an optimal transport map $T^*$, it is unique (almost everywhere), it is the gradient of a convex function, and it has the following form: 
       There exists $\bfw = (w_1,\ldots,w_n) \in \mathbb{R}^n$ such that
\begin{equation}
\label{def:OTmap}
T^*(x) = x_i \quad \textrm{if } 
x \in \operatorname{int}\Lag_i(\bfX,\bfw),
\end{equation}
where $\operatorname{int}$ denotes the interior. 
This defines $T^*$ uniquely almost everywhere because the boundaries of the cells form a set of measure zero.
In particular, 
% $\{ {T^*}^{-1}(\{ x_i \}) \}_{i=1}^n$ is a Laguerre tessellation  with $\vol(\Lag_i(\bfX,\bfw))=v_i \; \forall \; i$
$(T^*)^{-1}(\{ x_i \})$ is the Laguerre cell $\Lag_i(\bfX,\bfw)$ (up to a set of measure zero) with $\vol(\Lag_i(\bfX,\bfw))=v_i$ for all $i$.

The following theorem gives us a method to find the  weights $\bfw$ and consequently the optimal map \eqref{def:OTmap}. It goes back at least as far as \cite{Aurenhammer98} and is well-known in the optimal transport community.
        \begin{theorem}
    [See \cite{MerigotThibertOT}, Theorem 40 \& Proposition 37]
    \label{Thm:DualOT}
    Given $\bfX \in \distinct_n$ and $\bfv \in \mathbb{V}_n$, define the dual function $\mathcal{K}:\mathbb{R}^n \to \mathbb{R}$ by
            \[
            \mathcal{K}(\bfw) = 
            \int_\Omega \min_i (|x-x_i|^2 - w_i) \,\mathrm{d}x + \sum_{i=1}^n w_i\tvoli_i =
            \sum_{i=1}^{n} \int_{\Lag_i(\bfX,\bfw)} (|x-x_i|^2 - w_i) \,\mathrm{d}x + \sum_{i=1}^n w_i\tvoli_i.
            \]
            Then $\mathcal{K} \in C^1(\mathbb{R}^n)$, it is concave, and
            its gradient is given by
            \[
            \frac{\partial \mathcal{K}}{\partial w_i} (\bfw) = \tvoli_i - \cellvol_i(\bfX,\bfw) \quad \forall \, i \in \{1,\ldots,n\}.
            \]
            Moreover 
            \[
            \min_{T \in \mathcal{A}(\bfX,\bfv)}\mathcal{M}(T) = \max_{\bfw \in \mathbb{R}^{n}} \mathcal{K}(\bfw).
            \]
            In particular, the maximum of $\mathcal{K}$ is achieved.
            Let $w^* \in \mathbb{R}^n$ be a maximiser of $\mathcal{K}$. Then
            %\[
 %T_{\bfw^*}(x) := \underset{i \in \{1,\ldots,n\}}{\mathrm{argmin}} \; |x-x_i|^2 - w^*_i
 %           \]
$T_{\bfw^*}:\Omega \to \{ x_1,\ldots,x_n\}$ defined almost everywhere by
 \[
 T_{\bfw^*}(x) := x_i \quad \textrm{if } 
x \in \operatorname{int}\Lag_i(\bfX,\bfw^*)
\]
 is the optimal transport map for \eqref{def:OTproblem}. Furthermore,
            \[
            \vol(\Lag_i(\bfX,\bfw^*)) = v_i \quad \forall \, i \in \{1,\ldots,n\}.
            \]
        \end{theorem}
        In summary,  $\bfw^*$ maximises $\mathcal{K}$ if and only if the Laguerre cells $\Lag_i(\bfX,\bfw^*)$
        have the desired volumes $v_i$. The optimal weight vector, $\bfw^*$,  is not unique since we can add a constant vector to $\bfw^*$ and the cells are unchanged. However, if we fix a value for one of the weights we have uniqueness. Usually we chose $w_n = 0$ since this is easy to implement computationally.
        From now on, we let  $\bfwstarX$ denote the unique optimal weight vector with $w_n^*(\bfX;\bfv)=0$. By definition (see \eqref{eq:m xi}), 
        \[
\cellvol_i(\bfX,\bfwstarX) = \tvoli_i 
\quad \forall \, i \in \{1,\ldots,n\}.
\]
We define $\volcell_i(\bfX; \bfv)$
to be the $i^{\text{th}}$ Laguerre cell in the tessellation with seeds $\bfX$ and cell volumes $\bfv$, and $\cent_i(\bfX;\bfv)$ to be its centroid:
\begin{align} 
\label{defn:LaguerreWstar}
\volcell_i(\bfX; \bfv) & = \Lag_i(\bfX,\bfwstarX),  
\\          
\label{defn:CentWstar}
\cent_i(\bfX;\bfv) & = \cellcent_i(\bfX;\bfwstarX).  
\end{align}
 In this notation, by Theorem \ref{Thm:DualOT},
\begin{equation}
\label{eq:W2 in terms of Laguerre}
W_2^2(\mathcal{L}^d_\Omega,\nu(\bfX,\bfv)) = \sum_{i=1}^n \int_{L_i (\bfX; \bfv)} |x-x_i|^2 \, \text{d}x.
\end{equation}
        
        The state-of-the-art method for maximising $\mathcal{K}$ is the damped-Newton method, which can be found in \cite{Kitagawa_Merigot_Thibert_2019} and \cite[Algorithm 4]{MerigotThibertOT}. This is implemented for example in the software libraries geogram \cite{Geogram}, MATLAB-SDOT \cite{MATLAB-SDOT} and pysdot \cite{pysdot}. We use pysdot for the simulations in Section \ref{sec: numerics}.

\section{Statement of the problems}
\label{sec:statement of problem}
Our first goal is to invert a Laguerre tessellation in the following sense. We wish to recover a Laguerre tessellation only from the knowledge of the volumes and centroids of its cells. To be precise, we consider the following inverse problem:

\begin{goal}[Inverting a Laguerre tessellation]
\label{goal: recovering}
Let $\bfX \in \distinct_n$ and $\bfw \in \mathbb{R}^n$. For $i \in \{1,\ldots,n\}$, define $v_i=
\cellvol_i(\bfX,\bfw)$ and $b_i=\cellcent_i(\bfX,\bfw)$ to be the volume and centroid of the Laguerre cell $\Lag_i(\bfX,\bfw)$. The goal is to recover 
the Laguerre cells $\Lag_i(\bfX,\bfw)$
%$\bfX$ and $\bfw$ 
given only $\bfv=(v_1,\ldots,v_n)$ and $\bfB=(b_1,\ldots,b_n)$.
\end{goal}
We prove that this problem has a unique solution in Theorem \ref{thm:uniqueness}, in the sense that the Laguerre cells $\Lag_i(\bfX,\bfw)$ are unique. The generators $(\bfX,\bfw)$ 
are not unique because (by Lemma \ref{Lemma:Meyron})  
the seeds can be uniformly translated and dilated and a constant can be added to the weights without changing the Laguerre cells. 
 We also give a constructive method for recovering
 $\Lag_i(\bfX,\bfw)$ from
$\vB$ by solving an unconstrained convex optimisation problem; see Theorem \ref{cor:soln to Goal 1}. This is illustrated numerically in Section \ref{sec: numerics - diagram exists}. We consider the generalisation of inverting \emph{anisotropic} Laguerre tessellations in Section \ref{sec:aniso}.

\smallskip

% Moreover, if one of the seeds is fixed (to factor out the translation), if the norm of one of the other seeds is fixed (to factor out the dilation), and if one of the weights is fixed (to factor out the constant), then the generators $(\bfX,\bfw)$ are also unique. THIS IS FALSE by Mason's example of a cell in a corner with only one neighbour.
 
Our second goal is, given some target volumes and centroids (not coming from a Laguerre tessellation), find a Laguerre tessellation that best fits this data. To state the problem precisely we need some more notation.   
        Let $\mathcal{D}_n$ denote the set of  target volumes and centroids (or barycentres):
        \begin{multline*}
    \mathcal{D}_n=\left\{ (\bfv,\bfB) = ((v_1,\ldots,v_n),(b_1,\ldots,b_n))\in \mathbb{V}_n \times \Omega^n \, : \phantom{\sum_{i=1}^n} \right. \\
    \left.
    \sum_{i=1}^n v_i = \text{vol}(\Omega), \quad 
    \sum_{i=1}^n v_i b_i = \text{vol}(\Omega) \sigma(\Omega) \right\}.
    \end{multline*}
    The constraints on $\vB$ stem from \eqref{eq:MassOmega} and \eqref{eq:MassCentsOmega}.
They are necessary for there to exist a Laguerre tessellation with cells with volumes $v_i$ and centroids $b_i$.
    We refer to the elements $\vB \in \mathcal{D}_n$ as \emph{compatible data} for the domain $\Omega$.

    Given a domain $\Omega$ and a compatible set of target volumes and centroids 
        $\vB \in \mathcal{D}_n$, the goal is to find a Laguerre tessellation with cells of volume $v_i$ and minimal centroid error. To be precise, 
        we consider the following data-fitting problem:
\begin{goal}[Fitting a Laguerre tessellation] 
\label{goal: fitting}
Let $(\mathbf{v},\mathbf{B}) \in \mathcal{D}_n$. The goal is to solve the nonlinear least squares problem
        \begin{equation}
        \label{eq:NLS}
            \tag{NLS}
     \inf_{\bfX \in \mathbb{D}_n}  
f(\bfX;\vB)
           \end{equation}
           where $f:\mathbb{D}_n\to\mathbb{R}$ is defined by
           \begin{equation}
           \label{eq:f}
f(\bfX;\vB) = \sum_{i=1}^{n}\tvoli_i^2|\cent_i(\bfX; {\bf{v}}) - {{\tcenti_i}}|^2.
           \end{equation}
        \end{goal}

Note that the objective function can be expressed in terms of first moments of the Laguerre cells:
\[
f(\bfX;\vB)
%\sum_{i=1}^{n}\tvoli_i^2|\cent_i(\bfX; {\bf{v}}) - {{\tcenti_i}}|^2
= \sum_{i=1}^n \left|
\int_{L_i(\bfX;\bfv)} x \, \text{d} x  - v_i b_i
\right|^2.
\]
We choose to fit the zeroth moments (the volumes) exactly and to approximately fit the first moments. In other words, we choose the weights $\bfw$ so that the Laguerre cells have the desired volumes, and we choose the seeds $\bfX$ to minimise the centroid error. Alternatively, instead of fitting the volumes exactly, one could minimise a weighted sum of the volume and centroid error over $(\bfX,\bfw)$, as in \cite{PetrichEtAl2019}.
One advantage of our approach is that, for some data $\vB$, \eqref{eq:NLS} turns out to be essentially equivalent to 
a convex optimisation problem, as we will see in Section \ref{sec:rewriting NLS}.
One disadvantage of our approach is that the set $\distinct_n$ is not compact and so the optimisation problem \eqref{eq:NLS} may not have a solution; the infimum may not be attained. For some data $\vB$, infimising sequences $(\bfX_k)$ in $\distinct_n$ converge to a point in $\mathbb{R}^{nd} \setminus \distinct_n$ with non-distinct seeds. We illustrate this point numerically in Section \ref{sec:fitting}. We give a practical example of fitting a Laguerre tessellation to data coming from an EBSD image of steel in Section \ref{sec:EBSD}.

We also give necessary (but not sufficient) conditions on the data $\vB$ for there to exist $\bfX \in \distinct_n$ such that $f(\bfX;\vB) = 0$, that is, for there to exist a Laguerre diagram with cells of volumes $\bfv$ and centroids $\bfB$; see Section \ref{sec:existence}.

\section{Main results}
\label{sec:main results}
%In this section we study the existence, uniqueness and characterisation of solutions of \eqref{eq:NLS}, and recast it as a convex optimisation problem.

\subsection{The helper function $H$}
\label{subsec:H}

We introduce a concave function $H$ with the property that $f=|\nabla H|^2$. We will use $H$ to construct solutions to Goals 1 and 2.
Given $\vB \in \mathcal{D}_n$,    define $H:\RR^{nd} \rightarrow \RR$ by
    \begin{equation}
    \label{eq:H}
    H(\bfX) = H(\bfX;\vB) = F(\bfX;\bfv) - \frac{1}{2}\sum_{i=1}^{n} v_i |x_i|^2 + \sum_{i=1}^n v_i x_i \cdot b_i - \frac{1}{2}\int_{\Omega} |x|^2 \, \text{d}x,
    \end{equation}
    where $F:\mathbb{R}^{nd} \to \mathbb{R}$ is defined by
    \begin{equation}
    \label{eq:F}
        F(\bfX) = F(\bfX;\bfv) = 
        {\frac 12}
W_2^2(\ldomega,\nu(\bfX,\bfv)),
     \end{equation}
    and $W_2$ was defined in \eqref{eq:W2}. Note that if the triple $(\bfX,\bfv,\bfB)$ is such that $b_i$ lies in the Laguerre cell $L_i(\bfX;\bfv)$ for all $i$, then $H$ can be written as the difference of two transport costs up to a constant:
\[
H(\bfX)={\frac 12}W_2^2(\ldomega,\nu(\bfX,\bfv))-\frac 12 W_2^2\left(\nu(\bfB,\bfv),\nu(\bfX,\bfv)\right)-\frac{1}{2}\int_\Omega |x|^2\, \text{d}x+\frac{1}{2}\sum_{i=1}^n v_i|b_i|^2.
\]

First we recall a fact about $F$.

\begin{proposition}[Derivative of $W_2^2$; cf.~\cite{MerigotSantambrogioSarrazin}, Proposition 1] 
\label{Thm:MSS}
The function $F:\mathbb{R}^{nd} \to \mathbb{R}$ is semi-concave in the sense that the function 
$\bfX \mapsto F(\bfX) -  {\frac 12} \sum_{i=1}^n v_i | x_i|^2 $
is concave. Moreover, $F$ is continuously differentiable on the set of distinct seeds, $F \in C^{1}(\distinct_n)$, and
        \[
        \frac{\partial F}{\partial x_i}(\bfX) = v_i( x_i - \cent_i(\bfX;\bfv)) \quad \forall \; i \in \{1,\ldots,n\}. 
        \]
        \end{proposition}

The previous result is crucial in proving the following. 

\begin{theorem}[Properties of $H$]
\label{thm:properties H}
Let  $\vB \in \mathcal{D}_n$ and
 $H:\mathbb{R}^{nd} \to \mathbb{R}$ be the function defined in equation \eqref{eq:H}.
\begin{enumerate}[label=(\roman*)]
\item \label{Thrm4.2H:cents} If $\bfX \in \mathbb{D}_n$, then
\begin{equation}
\label{eq:Hcentroids}
    H(\bfX) = \sum_{i=1}^{n} v_i (\tcenti_i - \cent_i(\bfX;\bfv)) \cdot x_i.
    \end{equation}
In particular, if $\bfX \in \mathbb{D}_n$ generates a Laguerre tessellation with cells of volume $v_i$ and centroids $b_i$, then $H(\bfX)=0$.
\item \label{Thrm4.2H(t)=0} Let ${x_1} \in \mathbb{R}^{d}$ and ${\bfX = (x_1,\ldots,x_1)} \in \mathbb{R}^{nd}$. Then $H({\bfX})=0$.
\item \label{Thrm4.2Hconcave} $H$ is concave.
\item \label{Thrm4.2:HisC1} $H$ is continuously differentiable on the set of distinct seeds, $H \in C^1(\mathbb{D}_n)$, and for all $\bfX \in \mathbb{D}_n$ the gradient of $H$ is given by
\[
\frac{\partial H}{\partial x_i}(\bfX) = v_i(b_i - c_i(\bfX;\bfv)) \quad \forall \; i \in \{1,\ldots,n\}.
\]
In particular, $\nabla H(\bfX) = 0$ if and only if 
$\bfX$ generates a Laguerre tessellation with cells of volume $v_i$ and centroids $b_i$.
Moreover, the objective function $f$ defined in \eqref{eq:f} is related to $H$ by
\begin{equation}
\label{eq:f=|grad H|^2}
f(\bfX) = | \nabla H(\bfX) |^2.
\end{equation}
\item \label{Thrm4.2AffineTrans} Let $\bfX \in \mathbb{R}^{nd}$, $\lambda \ge 0$, $t \in \mathbb{R}^d$, $\mathbf{T}=(t,\ldots,t) \in \mathbb{R}^{nd}$. Then
\[
H(\lambda \bfX + \mathbf{T}) = \lambda H(\bfX).
\]
In particular, $H$ is 1-positively homogeneous and, for all $\bfX \in \distinct_n$,
\begin{equation}
\label{eq:invariance grad H}
\nabla H(\lambda \bfX + \mathbf{T})
= \nabla H(\bfX). 
\end{equation}
Moreover, for all $\lambda \in \mathbb{R}$ and $\bfX \in \mathbb{R}^{nd}$,
\[
H(\lambda \bfX) \le \lambda
H(\bfX)
\]
with equality if $\lambda \ge 0$.
\item \label{Thrm4.2Hsuperlinear} $H$ is superlinear, that is, for all $\bfX,\bfY \in \mathbb{R}^{nd}$,
\[
H(\bfX + \bfY) \ge H(\bfX) + H(\bfY).
\]
\item \label{Thrm4.2C=BthenGradH=0} If there exists $\bfX^* \in \distinct_n$ such that $c_i(\bfX^*;\bfv) = b_i$ for all $i \in \{1,\ldots,n\}$, then
\[
\max_{\bfX \in \mathbb{R}^{nd}} H(\bfX) = H(\bfX^*) = 0.
\]
{If not, then either 
\[
\sup_{\bfX \in \mathbb{R}^{nd}} H(\bfX)  = + \infty
\]
or
\[
\max_{\bfX \in \mathbb{R}^{nd}} H(\bfX)=0 = H((x_1,\ldots,x_1))
\]
for any $x_1 \in \mathbb{R}^d$.}

\item  
\label{Thrm4.2nodiagram}
Let $K \subset \mathbb{R}^{nd}$ be compact. If there does not exist a Laguerre tessellation with cells of volume $v_i$ and centroids $b_i$, then 
% \db{the maximum of $H$ over $K$ is achieved on $\partial K$ (possibly at a point in $ \partial K \setminus \distinct_n$ where the seeds are not all distinct).}
either the maximum of $H$ over $K$ is achieved on $\partial K$ or it is achieved on the set of non-distinct seeds $K \setminus \distinct_n$.
% \db{[Should this be something like the following: 
%  ``... then the maximum of $H$ over $K$ is achieved on $\partial K$ (possibly at a point in $ \partial K \setminus \distinct_n$ where the seeds are not all distinct)." I think the maximum could also be achieved in the interior of $K$ as well, but by dilation and part (vii) there is always a boundary maximiser?]}
%if there does not exist a Laguerre tessellation with cells of volume $v_i$ and centroids $b_i$, then 
%either the maximum of $H$ is not achieved on $\mathbb{R}^{nd}$ or it is achieved on the set of non-distinct seeds.
\end{enumerate}
\end{theorem}

%\newpage
\begin{proof} $\quad$
\begin{enumerate}[label=(\roman*)]
%\vspace{1em}
\item The first claim can be proved via a simple algebraic manipulation as follows. {By definition of $H$ and equation \eqref{eq:W2 in terms of Laguerre},}
\begin{align*}
\nonumber H(X) 
& = \frac{1}{2} \sum_{i=1}^{n} \int_{L_i(\bfX;\bfv)} |x-x_i|^2\, \text{d}x - \frac{1}{2}\sum_{i=1}^{n}  v_i|x_i|^2 + \sum_{i=1}^{n} v_ix_i \cdot \tcenti_i - \frac{1}{2}\int_{\Omega} |x|^2 \, \text{d}x  \\
& = \sum_{i=1}^{n} v_i(\tcenti_i - \cent_i(\bfX;\bfv))\cdot x_i. 
\end{align*}
\item Observe that if $\bfX = (x_1,\ldots,x_1)$ for $x_1 \in \RR^d$, then 
    \begin{align*}
        H({\bfX})
        & = \frac{1}{2}\int_{\Omega}|x-{x_1}|^2 \, \text{d}x - \frac{1}{2}\sum_{i=1}^{n}v_i|{x_1}|^2 + {\sum_{i=1}^n v_i x_1 \cdot b_i} -\frac{1}{2}\int_{\Omega}|x|^2\, \text{d}x\\
        & =
        {x_1 \cdot \left( -\vol(\Omega) \sigma(\Omega) 
        + \sum_{i=1}^n 
        v_i b_i \right)} 
        \\
        & = 0,
    \end{align*}
    where the last two equalities follow from {the assumption that} $(\bfv,\bfB) \in \mathcal{D}_n$. Alternatively, one can deduce that $H(\bfX)=0$ by letting $\lambda \to 0^+$ in part \ref{Thrm4.2AffineTrans}.
\item \label{Thrm4.2concave}

The semi-concavity of $F$ is stated in \cite[Proposition 1]{MerigotSantambrogioSarrazin} and proved in \cite[Propostion 21]{MerigotMirebeau}. The concavity of $H$ is an immediate consequence.

\item Proposition \ref{Thm:MSS} implies that $H$ is
%being 
continuously differentiable in $\distinct_n$ and immediately gives the expression for its gradient. % follows  
%from Theorem 4.1. 
%The rest of the result is a consequence of the expression.

\item Let $\bfX,\lambda,t$ and $\mathbf{T}$ be defined as in the statement of the theorem. Then, for all $i \in \{1,\ldots,n\}$, $L_i(\bfX;\bfv) = L_i(\lambda \bfX + \mathbf{T}; \bfv)$ by Lemma \ref{Lemma:Meyron} {and the uniqueness of optimal transport maps}. Consequently, if $X \in \distinct_n$ and $\lambda > 0$, then
\begin{equation*}
    H(\lambda \bfX + \mathbf{T}) \stackrel{\eqref{eq:Hcentroids}}{=} \sum_{i=1}^{n}v_i  (b_i - c_i(\lambda\bfX+\mathbf{T};\bfv)) \cdot (\lambda x_i + t) = \sum_{i=1}^{n} v_i (b_i - c_i(\bfX;\bfv)) \cdot (\lambda x_i + t).
\end{equation*}
Recall that $(\mathbf{v},\mathbf{B}) \in \distinct_n$ implies that $\sum_{i=1}^{n} v_ib_i = \text{vol}(\Omega) {\sigma}(\Omega) = \sum_{i=1}^{n} v_ic_i(\bfX;\bfv)$. Therefore 
$t \cdot \sum_{i=1}^{n}v_i (b_i - c_i(\bfX;\bfv)) = 0$
and so
%therefore, $t \cdot \sum_{i=1}^{n}v_i(b_i - c_i(\bfX;\bfv)) = 0$ giving, 
    \begin{equation*}
        H(\lambda \bfX + \mathbf{T}) = \sum_{i=1}^{n} v_i  (b_i - c_i(\bfX;\bfv)) \cdot \lambda x_i = \lambda \sum_{i=1}^{n} v_i (b_i - c_i(\bfX;\bfv)) \cdot x_i \stackrel{\eqref{eq:Hcentroids}}{=} \lambda H(\bfX),
    \end{equation*}
    which establishes the result for $\bfX \in \distinct_n$. The general result for $\bfX \in \mathbb{R}^{nd}$ and $\lambda=0$ can be established from the continuity of $H$. The expression $\nabla H(\bfX) = \nabla H(\lambda \bfX + \mathbf{T})$ follows immediately from the chain rule. Finally, to prove that $H(\lambda \bfX) \le \lambda
H(\bfX)$ for all $\lambda \in \mathbb{R}$, it suffices to show that $H(-\bfX) \leq -H(\bfX)$ since then, for all $\lambda < 0$, 
\[
{H(\lambda \bfX) = H((-\lambda)(-\bfX))
= -\lambda H(-\bfX)
\le \lambda H(\bfX).
}
\]
Since {$H(\mathbf{0})=0$ (by \ref{Thrm4.2H(t)=0}) and} $H$ is concave we have, 
\begin{equation*} 
0 = H(\mathbf{0}) = H \left( \frac{1}{2}\mathbf{X} + \frac{1}{2}(-\bfX) \right) \geq \frac{1}{2}H(\bfX) + \frac{1}{2}H(-\bfX)
\quad 
\Longleftrightarrow
\quad  H(-\bfX) \leq -H(\bfX),
\end{equation*}
%rearrange and dividing through gives, 
%\begin{equation*}
%    H(-X) \leq -H(X),
%\end{equation*}
as required.
%To conclude, apply positive homogeneity to $-\lambda$ if $\lambda < 0$. 
    \item Let $\bfX,\bfY \in \RR^{nd}$. Then {by positive homogeneity and concavity},
    \begin{equation*} 
    \frac{1}{2}H(\bfX+\bfY) = H \left( \frac{1}{2} \bfX + \frac 12 \bfY \right) \ge %\frac 12 H(\bfX)+ \frac 12 H(\bfY) 
    \frac{1}{2}(H(\bfX)+H(\bfY)).
    \end{equation*}
    Multiplying both sides by 2 gives the desired result.
 \item If there exists $\bfX^{*} \in \distinct_n$ such that $c_i(\bfX^*;\bfv) = b_i$, then $\nabla H(\bfX^*) = 0$ by \ref{Thrm4.2:HisC1}. Since $H$ is concave, it follows that $X^{*}$ is a global maximiser of $H$. Furthermore, $H(\bfX^{*}) = 0$ by \ref{Thrm4.2H:cents}.
  If there does not exist such an $\bfX^* \in \distinct_n$, then there are two cases. Either there exists $\bfX \in \mathbb{R}^{nd}$ such that $H(\bfX) >0$, in which case $\lim_{\lambda \to \infty} H(\lambda \bfX) = + \infty$. Or, if $H(\bfX) \le 0$ for all $\bfX \in \mathbb{R}^{nd}$, then the maximum value of $H$ is zero, which is achieved when all of the seeds are in the same place by \ref{Thrm4.2H(t)=0}. 

\item 
If there does not exist $\bfX^{*} \in \distinct_n$ such that $c_i(\bfX^*;\bfv) = b_i$, then
$\nabla H$ is non-zero everywhere in $\distinct_n$. Therefore the maximum value of $H$ on the compact set $K$ is achieved either on the boundary of $K$ or otherwise at a point in the interior of $K$ where $H$ is not differentiable, namely at a point in $K \setminus \distinct_n$.
%In the second case we can take $\bfX \in \partial K \setminus \distinct_n$ by the dilation-invariance of $H$, property \ref{Thrm4.2AffineTrans}.
\end{enumerate}
\end{proof}

\begin{remark}[Consequences of Theorem \ref{thm:properties H} for Goal \ref{goal: fitting}]
\label{remark:interp}

 Minimising the least squares error $f$ is equivalent to minimising $|\nabla H|$ by equation \eqref{eq:f=|grad H|^2}. Assume that $\inf_{\distinct_n} f$ exists. Then, due to the invariance of $\nabla H$ under uniform translations and dilations of the seeds (equation \eqref{eq:invariance grad H}), we can restrict the minimisation problem \eqref{eq:NLS} to a compact subset of $\distinct_n$:
\[
 \min_{\bfX \in \mathbb{D}_n} 
f(\bfX)
=
 \min_{\bfX \in K \cap\, \mathbb{D}_n} 
f(\bfX)= \min_{\bfX \in K \cap \,\mathbb{D}_n} 
| \nabla H(\bfX)|^2
\]
for any compact set $K \subset \mathbb{R}^{nd}$ containing a neighbourhood of $0$, or more generally a neighbourhood of a point of the form $(t,\ldots,t) \in \mathbb{R}^{nd}$, $t \in \mathbb{R}^d$. 
For a suitable choice of compact convex set $K$, it turns out that maximising $H$ on $K$ (which is a convex optimisation problem) is equivalent to locally minimising $f=|\nabla H|^2$ on $K \cap \distinct_n$; see Section \ref{sec:rewriting NLS}.
% In the very special case where there exists $\bfX^* \in \distinct_n$ such that the Laguerre cells $L_i(\bfX^*;\bfv)$ have exactly volumes $v_i$ and centroids $b_i$, then $\bfX^*$ is a maximiser of the concave function $H$ (by Theorem \ref{thm:properties H}(iii),(iv)). Therefore $\bfX^*$ is easy to compute numerically using convex optimisation.
% We examine this special case in Sections \ref{sec:uniqueness} and \ref{sec:existence}. In Section \ref{sec:rewriting NLS} we return to the general case of approximately fitting a diagram to data. We will find conditions under which maximisers of $H$ on a compact set are local minimisers of $f$. Since $H$ is concave, this essentially reduces solving \eqref{eq:NLS} to a convex optimisation problem, which is numerically tractable.
\end{remark}

% \begin{remark}[Justification of the heuristic from \cite{TefraRowenhurtsDirectParameterEstimateforGPBD}]
% \label{remark:heuristic}
% A popular heuristic for fitting Laguerre diagrams to volume and centroid data is to take the seeds to be the target centroids, $\bfX=\bfB$; see for example \cite{TefraRowenhurtsDirectParameterEstimateforGPBD}. We give a partial justification for this heuristic. We can decompose $H$ as
% \[
% H(\bfX) = F(\bfX) + G(\bfX), \quad \textrm{where} \quad G(\bfX) = - \frac{1}{2}\sum_{i=1}^{n} v_i |x_i|^2 + \sum_{i=1}^n v_i x_i \cdot b_i - \frac{1}{2}\int_{\Omega} |x|^2 \, \text{d}x
% \]
% \db{and where $F$ was defined in \eqref{eq:F}.} 
% Consider maximising $H$ over  the compact set $\Omega^n$ 
% (which is related to minimising the fitting error $f$ by the remark above). The maximiser depends on the competition between the terms $F$ and $G$. It can be shown that the first term, $F$, is maximised when all the seeds are equal and lie at the point in 
% $\Omega$ that is furthest from the centroid of $\Omega$. On the other hand, the second term, $G$, is maximised by $\bfX = \bfB$, namely when the seeds are the target centroids. If the second term dominates, then we would expect the maximiser of $H$ (and the minimiser of $f$) to be close to $\bfB$. This partially justifies the heuristic used by \cite{TefraRowenhurtsDirectParameterEstimateforGPBD} and other authors. \db{On the other hand, if the first term dominates, then the maximiser of $H$ may not belong to $\distinct_n$.} 
% \end{remark}

\subsection{Goal \ref{goal: recovering}: Uniqueness and construction of the solution}
\label{sec:uniqueness}

\begin{definition}[Compatible diagrams]
Given $(\bfv,\bfB) \in \mathcal{D}_n$, we say a Laguerre diagram $\{L_i(\bfX,\bfv)\}_{i=1}^n$ 
%generated by $(\bfX,\wstar(\bfX,\bfv))$ 
is \emph{compatible with the data $(\bfv,\bfB)$} if $c_i(\bfX;\bfv) = b_i$ for all $i \in \{1,\ldots,n\}$. In other words, the diagram is compatible with the data $(\bfv,\bfB)$ if the Laguerre cells have volumes $v_i$ and centroids $b_i$.
\end{definition}

The next result ensures us that if there exists a compatible diagram, then it is unique. Therefore Goal \ref{goal: recovering} has a unique solution (in the sense that the Laguerre cells are unique; the generators $(\bfX,\bfw)$ are not unique by Lemma \ref{Lemma:Meyron}).

\begin{theorem}[Uniqueness of compatible diagrams]
         \label{thm:uniqueness}
         Let $(\bfv,\bfB) \in \mathcal{D}_n$.
            Suppose $\bfX, \bfY \in \distinct_{n}$ are such that
            $\{ L_i(\bfX;\bfv) \}_{i=1}^n$ and $\{ L_i(\bfY;\bfv) \}_{i=1}^n$ are compatible with the data $(\bfv,\bfB)$, i.e,
            \begin{equation*}
                \cent_i(\bfX;\bfv) = \cent_i(\bfY;\bfv) = b_i 
                \quad \forall \; i \in \{1,\ldots,n\}.
            \end{equation*}
            Then
            \begin{equation*}
                L_i(\bfX;\bfv) = L_i(\bfY;\bfv)  \quad \forall \; i \in \{1,\ldots,n\}.
            \end{equation*}
            In other words, a Laguerre tessellation is uniquely determined by the volumes and centroids of its cells.
\end{theorem}
            
            \begin{proof}
                Let $T^{*}:\Omega \to \bfX$ denote the optimal transport map between $\ldomega$ and $\nu(\bfX, \bfv)$, i.e., $T(x)=x_i$ if $x \in \operatorname{int} L_{i}(\bfX ; \bfv)$. Define a map $S:\Omega \to \bfX$ by
                \[
                S(x) = x_i 
                \quad \textrm{if } x \in \operatorname{int} L_{i}(\bfY ; \bfv).
                \]
                Since $\ldomega(S^{-1}(\{x_i\})) = \ldomega(L_i(\bfY ;\bfv)) = v_i$, $S$ is admissible for the transport problem between $\ldomega$ and $\nu(X,\bfv)$. Moreover,
                \begin{align*}
                    M(S)
                    & = \int_{\Omega} |x-S(x)|^2 \, \text{d}x  
                    \\
                    & = \int_{\Omega} |x|^2 \, \text{d}x + \sum_{i=1}^{n}\int_{L_i({\bfY;\bfv})}  (|x_i|^2 - 2x \cdot x_i) \,\text{d}x 
                    \\
                    &= \int_{\Omega} |x|^2 \, \text{d}x + \sum_{i=1}^{n} (v_i|x_i|^2 - 2v_i \cent_i(\bfY;\bfv) \cdot x_i) 
                    \\ 
                    &= \int_{\Omega} |x|^2 \, \text{d}x + \sum_{i=1}^{n} (v_i|x_i|^2 - 2v_i \cent_i(\bfX;\bfv)  \cdot x_i) 
                    \\
                    & = \int_{\Omega} |x|^2 \, \text{d}x + \sum_{i=1}^{n}\int_{L_i({\bfX;\bfv})}  (|x_i|^2 - 2x \cdot x_i) \, \text{d}x 
                    \\
                    & = \int_{\Omega} |x-T(x)|^2 \, \text{d}x 
                    \\
                    & = M(T^{*}).
                \end{align*}
                Therefore $S$ has the same transport cost as the optimal map $T^*$. 
                Consequently, since the optimal transport map is unique (almost everywhere),  $L_{i}(\bfX ; \bfv)=L_{i}(\bfY ; \bfv)$ for all $i \in \{1,\ldots,n\}$, as required.
            \end{proof}

An immediate consequence of Theorems \ref{thm:properties H} and \ref{thm:uniqueness} is the following:
\begin{theorem}[Goal \ref{goal: recovering} is a convex optimisation problem]
\label{cor:soln to Goal 1}
Let $\bfX \in \distinct_n$ and $\bfw \in \mathbb{R}^n$. For $i \in \{1,\ldots,n\}$, define $v_i=
\cellvol_i(\bfX,\bfw)$ and $b_i=\cellcent_i(\bfX,\bfw)$ to be the volume and centroid of the Laguerre cell $\Lag_i(\bfX,\bfw)$.
Let $\bfv=(v_1,\ldots,v_n)$ and $\bfB=(b_1,\ldots,b_n)$.
Then the concave function $H(\, \cdot \; ;\vB)$ achieves its maximum over $\mathbb{R}^{nd}$ in $\distinct_n$. 
Let $\bfY \in \distinct_n$ be any maximiser. Then 
\[
L_i(\bfY;\bfv) = L_i(\bfX;\bfv)  \quad \forall \; i \in \{1,\ldots,n\}.
\]
In other words, we can recover a Laguerre tessellation from the volumes and centroids of its cells by solving a convex optimisation problem (maximising the concave function $H$ in $\mathbb{R}^{nd}$).
\end{theorem}
\begin{proof}
Note that $\Lag_i(\bfX,\bfw)=L_i(\bfX;\bfv)$ and $\xi_i(\bfX,\bfw)=\cent_i(\bfX;\bfv)$.
The function $H(\, \cdot \; ; \vB)$ achieves its maximum over $\mathbb{R}^{nd}$ at $\bfX \in \distinct_n$ because $H$ is concave and $\nabla H(\bfX ;\vB)=0$ by Theorem \ref{thm:properties H} \ref{Thrm4.2:HisC1}. 
If $\bfY \in \distinct_n$ is any maximiser, then $\nabla H(\bfY ;\vB)=0$ and hence $\cent_i(\bfY;\bfv) = b_i = \cent_i(\bfX;\bfv)$.
We conclude from Theorem \ref{thm:uniqueness} that
$L_i(\bfY;\bfv) = L_i(\bfX;\bfv)$ for all $i \in \{1,\ldots,n\}$, as required.
\end{proof}

\subsection{Goal \ref{goal: recovering}: Existence of compatible diagrams}

\label{sec:existence}

In this section we seek necessary conditions on the data $\vB$ for the existence of a compatible diagram, namely, for the minimum value of the fitting error \eqref{eq:f} to be zero. 
In this case the compatible diagram can be found by maximising the concave function $H$, as described in the previous section.
In Section \ref{sec:rewriting NLS} we will study Goal \ref{goal: fitting}, where the minimum of the fitting error \eqref{eq:f} is greater than zero.

Generically, for `typical' data $(\bfv,\bfB) \in \mathcal{D}_n$ there does not exist a compatible diagram with cells with volumes $v_i$ and centroids $b_i$, as the following example demonstrates.

\begin{example}[Non-existence of a compatible diagram]
\label{example: non-existence}
Take $d=1$, $\Omega=[0,1]$, $n=2$, $\bfv=(0.5,0.5)$, $\bfX=(x_1, x_2) \in \mathbb{R}^2$ with $x_1 < x_2$. Then $L_{1}(\bfX ; \bfv)=[0,0.5]$ and 
$L_{2}(\bfX ; \bfv)=[0.5,1]$. (Note that this is true whatever the positions of $x_1$ and $x_2$, as long as $x_1 < x_2$.) If the target centroids are $\bfB=(0.25,0.75)$, then the diagram $\{ L_{i}(\bfX ; \bfv)\}_{i=1}^2$ is compatible with $(\bfv,\bfB)$ but, for any other choice of $\bfB$, there does not exist a compatible diagram.
\end{example}
    \begin{lemma}[Necessary condition for existence: minimum distance between centroids]
    Let $(\bfv,\bfB) \in \mathcal{D}_n$ and  
        suppose that there exists a Laguerre tessellation $\{ L_i(\bfX;\bfv) \}_{i=1}^n$ of $\Omega$ that is compatible with the data $(\bfv,\bfB)$. 
        Then, for all $i \in \{1,\ldots,n\}$,
        \begin{equation}
        \label{eq: distance from boundary}
\mathrm{dist}(\cent_i(\bfX;\bfv),\partial \Omega)
=
\mathrm{dist}(b_i,\partial \Omega)
\ge 
r_i,
    \end{equation}
where
\[
r_i = 
\frac{v_i}{4 \alpha_{d-1}} \left( \frac{\mathrm{diam}(\Omega)}{2} \right)^{1-d}
%\frac 12 \mathrm{diam}(\Omega)^{1-d} v_i.
\]
and $\alpha_{d-1}$ is the volume of the unit ball in $\mathbb{R}^{d-1}$.
Moreover, for all $i,j \in \{ 1,\ldots, n\}$, $i \ne j$,
\begin{equation}
        \label{eq: distance between centroids}
            |\cent_i(\bfX;\bfv) - \cent_j(\bfX;\bfv)| = |b_i - b_j | \geq r_{i} + r_j.
        \end{equation}
    \end{lemma}
%The lower bounds \eqref{eq: distance from boundary}, \eqref{eq: distance between centroids} are sharp when $d=1$, as Example \ref{example: non-existence} illustrates: 
%\[
%\mathrm{dist}(\cent_1(\bfX;\bfv),\partial \Omega)
%    =
%    \mathrm{dist}(\cent_2(\bfX;\bfv),\partial \Omega) = 0.25 = r_1 = r_2
%\]
%and
%\[
%    |\cent_1(\bfX;\bfv)-\cent_2(\bfX;\bfv)| = |0.25-0.75|=0.5 = r_1 + r_2.
%\]    
    \begin{proof}
   The idea for the proof comes from \cite[Appendix A]{MerigotSantambrogioSarrazin}.
        To prove \eqref{eq: distance from boundary} and \eqref{eq: distance between centroids},
        it suffices to show that
        \begin{equation*}
            \mathrm{dist}(b_i,\partial L_i(\bfX;\bfv) )\geq r_i
        \end{equation*}
        for all $i \in \{1,...,n\}$. 
        %Let $j \in \{1,\ldots,n\}$, $j \ne i$, satisfy
        %\[
        %\mathrm{dist}(b_i,\partial L_i(\bfX;\bfv) ) =
    %\mathrm{dist}(b_i, L_i(\bfX;\bfv) \cap L_j(\bfX;\bfv)).
        %\]
        %[But this does not include the case where the closest point to the boundary of the cell is on $\partial \Omega$.]
         Let $e_d$ be the $d$-th standard basis vector of $\mathbb{R}^{d}$.
        Without loss of generality (by rotating and translating $\Omega$ if necessary), we can assume that
        %\begin{align*}
            %L_i(\bfX;\bfv) \cap L_j(\bfX;\bfv) \subseteq \{x \in \RR^d: x \cdot e_d = 0 \}, \\
         \[
             L_i(\bfX;\bfv) \subseteq \{x \in \RR^d: x \cdot e_d \ge 0\}
             \]    
        %\end{align*}
      and
      \[
      \text{dist}(b_i,\partial L_i(\bfX;\bfv) ) = b_i \cdot e_d. 
      \]
Let $\Pi(\Omega)$ denote the projection of $\Omega$ onto the set $\{ x \in \mathbb{R}^d : x \cdot e_d = 0\}$. 
      There exists $s > 0$ such that    \begin{align*}
            \frac{v_i}{2} & = \mathcal{L}^d(L_i(\bfX;\bfv) \cap \{x \in \mathbb{R}^d : x \cdot e_d \le s\}) 
            \\
& \le s \mathcal{H}^{d-1}(\Pi(\Omega))
            \\
            & \le 
            s \alpha_{d-1} \left( \frac{\mathrm{diam}(\Omega)}{2} \right)^{d-1} 
        \end{align*}
        by the isodiametric inequality \cite[Theorem 8.8]{Gruber} and the fact that $\mathrm{diam}(\Pi(\Omega)) \le \mathrm{diam}(\Omega)$. By definition of $r_i$, we have $\frac s2 \ge r_i$.
                 Define 
                 \[
                 A = L_i(\bfX;\bfv) \cap \{x \in \mathbb{R}^d :  x \cdot e_d \ge s\}.
                 \]
                 Note that
                 $\mathcal{L}^d(A)=v_i/2$.
                 Then
        \begin{align*}
            b_i \cdot e_d = \frac{1}{v_i}\int_{L_i} x\cdot e_d\, \text{d}x \geq  \frac{1}{v_i}\int_{A} x \cdot e_d \, \text{d}x 
            \ge 
            \frac{1}{v_i}\int_{A} s \,  \text{d}x
            = \frac{s}{2} \geq {r_i},
        \end{align*}
    as desired. 
    \end{proof}

    %\db{They are not quite the same. Lemma 4.10 gives a sharper bound for the box domain. E.g., for $d=2$, Lemma 4.8 gives
    %\[
%\mathrm{dist}(b_i,\partial \Omega) \ge \frac 12 \frac{1}{\sqrt{L_1^2 + L_2^2}} v_i
 %   \]
  %  whereas Lemma 4.10 gives
  %  \[
  %  \mathrm{dist}(b_i,\partial \Omega) \ge \frac 12 \frac{1}{L_j} v_i
  %  \ge \frac 12 \frac{1}{\sqrt{L_1^2 + L_2^2}} v_i
  %  \]
  %  for $j \ne i$. However, Lemma 4.8 is more general, so I would keep both.
  %  }

        %\end{figure}
    If the domain $\Omega$ is a cuboid, then one can improve the bound \eqref{eq: distance from boundary} as follows.
    
%    'nice' one can quite easily obtain bounds by considering values of $X$ which make the corresponding tessellation, and consequently $H$, easy to compute.  In light of Theorem 4.2 \eqref{Thrm4.2C=BthenGradH=0} we can then check if $H(\bfX) \leq 0$, this approach gives restrictions on the values the centroids can take.
    \begin{lemma}[Necessary condition for existence: distance of centroids from boundary]
        Let $\Omega = [0,l_1] \times \cdots \times [0,l_d]$ and let $b_i^j$ denote the $j^{\text{th}}$ coordinate of the $i^{\text{th}}$ target centroid. Let $(\bfv,\bfB) \in \mathcal{D}_n$ and  
        suppose that there exists a Laguerre tessellation $\{ L_i(\bfX;\bfv) \}_{i=1}^n$  of $\Omega$ that is compatible with the data $(\bfv,\bfB)$. Then
\[
\frac{v_il_j}{2 \vol(\Omega)} \leq b_{i}^{j} 
\leq l_j-\frac{v_il_j}{2  \vol(\Omega)}.
\]
        In particular,
        \begin{equation*}
        \mathrm{dist}(\cent_i(\bfX;\bfv),\partial \Omega)
=
\mathrm{dist}(b_i,\partial \Omega)
\ge \min_j \frac{v_il_j}{2 \vol(\Omega)}.
        \end{equation*}
    \end{lemma}

\begin{proof}
    The inequalities are obtained by computing $H(\lambda \mathbf{e}^{j}_i)$ for suitable $\lambda \in \mathbb{R} \setminus \{0\}$, where $\mathbf{e}^{j}_i \in \mathbb{R}^{nd}$ denotes the $(d(i-1)+j)$-th standard basis vector of $\mathbb{R}^{nd}$, for $i \in \{1,...,n\}$ and $j \in \{1,\ldots,d\}$. If there exists a compatible diagram, then $H( \lambda  \mathbf{e}_i^{j}) \leq 0$ by Theorem \ref{thm:properties H}(vii). With this choice of seeds, $\bfX=\lambda  \mathbf{e}_i^{j}$, there are only two distinct seeds (all the seeds are located at the origin except for the $i$-th seed, which has coordinates $x_i^j=1$ and $x_i^k=0$ for $k \ne j$) and hence $H(\bfX)$ is easy to evaluate.

    Assume without loss of generality that $i = 1$ and $j = 1$. In order to evaluate $H(\lambda \mathbf{e}_1^1)$, we need to compute $W_2^2(\ldomega,\nu(\bfX,\bfv))$, where $\nu(\bfX,\bfv)) = v_1 \delta_{x_1} + (\vol(\Omega)-v_1) \delta_0$ and $x_1=(\lambda,0,\ldots,0) \in \mathbb{R}^d$.
    The optimal transport map $T^*$ partitions $\Omega$ into two cells, one of volume $v_1$ and the other of volume $\vol(\Omega) - v_1$. The boundary between the two cells consists of the points $x \in \Omega$ satisfying the equation
    \begin{equation*}
        |x- x_1|^2 - w \leq |x-0|^2
        \quad 
    \Longleftrightarrow
    \quad x \cdot (1,0,\ldots,0) = \frac  {\lambda^2 - w}{2 \lambda}
    \end{equation*}
    for some $w \in \mathbb{R}$.
    Therefore the boundary is contained in the plane with normal vector $(1,0,\ldots,0)$
    and distance $\delta:=(\lambda^2 -w)/(2 \lambda)$ from the origin. 
    The value of $\delta$ is determined by the volume constraints and the sign of $\lambda$. If $\lambda > 0$, 
    then the Laguerre cell corresponding to the seed $x_1$
    is $L_1 := [\delta,l_1] \times [0,l_2] \times \cdots \times [0,l_d]$, which has volume $ v_1$.
    Therefore 
    \[
(l_1-\delta) l_2 \cdots l_d = v_1
  \quad \Longleftrightarrow
  \quad 
  \delta = l_1 -  \frac{v_1}{l_2 \cdots l_d}.
    \]
    Note that $L_1$
    has centroid
    \[
c_1 := \left( \delta + \frac{l_1-\delta}{2}, \frac{l_2}{2},\ldots,\frac{l_d}{2} \right)
=
\left( l_1 - \frac{v_1}{2 l_2 \cdots l_d}, \frac{l_2}{2},\ldots,\frac{l_d}{2} \right).
    \]
Let 
$L_0 = [0,\delta] \times [0,l_2] \times \cdots \times [0,l_d]$ be the Laguerre cell corresponding to the seed $0$.
Then
\begin{align*}
H(\lambda \mathbf{e}_i^j)
& 
= \frac 12 
\int_{L_1} |x - x_1|^2 \, \text{d}x
+ \frac 12 \int_{L_0} |x-0|^2 \, \text{d}x 
 - \frac{1}{2}v_1 |x_1|^2 
 + v_1 x_1 \cdot b_1 - \frac{1}{2}\int_{\Omega} |x|^2 \, \text{d}x
\\ & =
v_1 x_1 \cdot \left( b_1 - c_1 \right)  
\\
& = v_1 \lambda \left(  b_1^1 - \left(  l_1 - \frac{v_1}{2 l_2 \cdots l_d}\right)
\right).
\end{align*}
Since $\lambda >0$,
\[
H(\lambda \mathbf{e}_1^1) \le 0 
\quad \Longleftrightarrow \quad 
b_1^1 \le l_1 - \frac{v_1}{2 l_2 \cdots l_d} =
l_1 - \frac{v_1 l_1}{2 \vol(\Omega)},
\]
as required. The other inequality, $b_1^1 \ge v_1 l_1 /(2 \vol (\Omega))$, can be derived analogously by taking $\lambda < 0$.    
\end{proof}

\begin{lemma}[Necessary condition for existence: cyclical monotonicity]
Let $(\bfv,\bfB) \in \mathcal{D}_n$ and  
        suppose that there exists a Laguerre tessellation $\{ L_i(\bfX;\bfv) \}_{i=1}^n$  of $\Omega$ that is compatible with the data $(\bfv,\bfB)$.         Then the set $\{(\tcenti_i,x_i)\}_{i=1}^n$ is cyclically monotone, which means that for every
index set $\mathcal{I} \subseteq \{1,\ldots,n\}$ and every
permutation $\sigma$ of $\mathcal{I}$,
\[
\sum_{i \in \mathcal{I}} b_i \cdot x_i 
\ge 
\sum_{i \in \mathcal{I}} b_i \cdot x_{\sigma(i)}.
\]
In particular, by taking $\mathcal{I}=\{i,j\}$, it follows that
\[
%b_i \cdot x_i + b_j \cdot x_j \ge 
%b_i \cdot x_j + b_j \cdot x_i
%\quad
%\Longleftrightarrow
%\quad 
(b_i - b_j) \cdot (x_i - x_j) \ge 0
\]
for all
 $i,j \in \{1,\ldots,n\}$, $i \ne j$.
\end{lemma}
\begin{proof}
 Let $T^{*}:\Omega \to \bfX$ denote the optimal transport map between $\ldomega$ and $\nu(\bfX, \bfv)$. A standard result in optimal transport theory states that $T^* = \nabla u$ for some convex function $u:\Omega \to \mathbb{R}$ \cite[Section 1.3.1]{SantambrogioBook}.
Since $b_i = \cent_i(\bfX;\bfv) \in L_i(\bfX;\bfv)$, then $T^*(b_i)=x_i$. Therefore
the set $\{ (b_i,x_i)\}_{i=1}^n = \{ (b_i,T^*(b_i)) \}_{i=1}^n = \{ (b_i,\nabla u(b_i)) \}_{i=1}^n$ is contained in the graph of the subdifferential of $u$. It follows that $\{ (b_i,x_i) \}_{i=1}^n$ is cyclically monotone by Rockafellar's Theorem \cite[p.~238, Theorem 24.8]{Rockafellar}.
\end{proof}
          
As far as we are aware, it is an open problem to find sufficient conditions for the existence of compatible diagrams, that is, to find conditions on $(\bfv,\bfB) \in \mathcal{D}_n$ that guarantee that there exists a Laguerre diagram that is compatible with $(\bfv,\bfB)$. A fully discrete version of this problem (where the Lebesgue measure on $\Omega$ is replaced by a discrete measure coming from a discretisation of $\Omega$) is solved in \cite[Theorem 2.1]{BriedenGritzmann2012}.

\subsection{Goal \ref{goal: fitting}: Rewriting \eqref{eq:NLS} 
as a convex optimisation problem}
\label{sec:rewriting NLS}
%This section is devote to relating the two optimisation problems, namely, the concave optimisation problem we encountered in Section 4.2 and the NLLS problem that arises from the data-fitting problem. We have already seen that $f(\bfX) = |\nabla H(\bfX)|^2$ and if there exists a diagram the two optimisation problems are equivalent. However, existence of a diagram is not guaranteed and we would like the problems to be related in more generality. This section is devoted to obtaining such results, in particular we show that maximisers of $H$ correspond to local minimisers of $f$ under certain assumptions on the topology. Here, when we say topology we mean the topology of the graph structure induced by the Laguerre Tesselation.

%Throughout this section we consider a general convex positively homogeneous function which we call $g$. The function $g$ is assumed to locally 3-times continuously differentiable which holds in a neighbourhood of a point if the topology of the graph structure remains unchanged in that neighbourhood. The precise statement of the result is below. The first result is below.
%\sr{If $g$ is general maybe it has nothing to do with a graph structure of a diagram. Do we need to say something else about $g$ so that the comment about graph structure makes sense? Or remove the comment about graph structure?}

%Add example to show that minimising $g$ is not equivalent to minimising $| \nabla g|$ in general. 

In this section we consider Goal \ref{goal: fitting}. Our main theorem is the following.

\begin{theorem}[Local minimisers of \eqref{eq:NLS} can be found by convex optimisation]
\label{thm:main}
% Old version:
% Let $(\bfv,\bfB) \in \mathcal{D}_n$ be such that there does not exists a Laguerre diagram compatible with $(\bfv,\bfB)$.
% Let \db{$t \in \mathbb{R}^d$, $\mathbf{T} = (t,\ldots,t) \in \mathbb{R}^{nd}$}, $R>0$, and 
% $B_R=\{\bfX \in \mathbb{R}^{nd} : |\bfX - \db{\mathbf{T}} | \le R\}$. 
% Assume that there exists 
% \db{$\bfX^* \in \distinct_n \cap \partial B_R$} such that $\bfX^*$ is a global maximiser of the concave function $H$ on $B_R$, $H$ is 3-times continuously differentiable in a neighbourhood of $\bfX^*$, and $\mathrm{dim}(\mathrm{ker}(D^2 H(\bfX^*)))=1+d$. 
% Then $\bfX^*$ is a local minimiser of $f$ on $\distinct_n$. In particular, we can find local minimisers of the least squares error $f$ by solving the convex optimisation problem $\max_{B_R} H$.
Let $(\bfv,\bfB) \in \mathcal{D}_n$ be such that there does not exists a Laguerre diagram compatible with $(\bfv,\bfB)$. 
Fix $t \in \mathbb{R}^d$, $\mathbf{T} = (t,\ldots,t) \in \mathbb{R}^{nd}$ and  $R>0$. Define $B \subset \mathbb{R}^{nd}$ to be the compact set
$B=\{\bfX \in \mathbb{R}^{nd} : |\bfX - \mathbf{T} | \le R\}$. 
Let $\bfX^* \in B$ be a global maximiser of the concave function $H$ on $B$.
Assume that 
$\bfX^* \in \distinct_n$, $H$ is 3-times continuously differentiable in a neighbourhood of $\bfX^*$, and $\mathrm{dim}(\mathrm{ker}(D^2 H(\bfX^*)))=1+d$. 
Then $\bfX^*$ is a local minimiser of $f$ on $\distinct_n$. In particular, we can find local minimisers of the least squares error $f$ by solving the convex optimisation problem $\max_{B} H$.
\end{theorem}

\begin{remark}[Assumptions of Theorem \ref{thm:main}]
\label{remark:assumptions}
The assumptions of Theorem \ref{thm:main} are implicit assumptions on the data $\vB$. We discuss each one in turn.

We only consider the case where there does not exists a Laguerre diagram compatible with $(\bfv,\bfB)$ because the other case has already been covered in Theorem \ref{cor:soln to Goal 1}.

% \db{The choice of $t$ and $R$ are arbitrary in the following sense.
% By the translation- and dilation-invariance of $f=|\nabla H|^2$ (see Theorem \ref{thm:properties H} \ref{Thrm4.2AffineTrans}), $\inf_{\distinct_n} f = \inf_{\distinct_n \cap B(t,R)} f$ for any choice of $t$ and $R$. Moreover, fixing the position of the first seed to be $x_1^*=t$ is not restrictive when it comes to maximising $H$ due to the translation-invariance of $H$.}
%(The choice of $R$ does affect the maximum value of $H$ on $B_R$, since $H$ is not dilation invariant, but it does not affect the minimum value of $f$.)
%The assumption that $\bfX^* \in \partial B_R$ is an immediate consequence of the assumption that $\bfX^* \in \distinct_n$ by Theorem \ref{thm:properties H} \ref{Thrm4.2nodiagram}. However, 

The assumption that $\bfX^* \in \distinct_n$, namely that the seeds are distinct, is restrictive and does not hold for all data $\vB \in \mathcal{D}_n$; see Section \ref{subsec: larger perts} for a numerical illustration. This is explained in part in Remark \ref{remark:heuristic}; the first term of $H$ (the optimal transport term) is maximised when all the seeds are in the same place. If the data $\vB$ is such that the first term of $H$ is the dominant term, then we  expect some of the seeds in $\bfX^*$ to coincide.

It can be shown using results from \cite{deGournay2018,Kitagawa_Merigot_Thibert_2019} that $H$ is 2-times continuously differentiable in $\distinct_n$. 
Expressions for the derivatives of the volumes and centroids of Laguerre cells with respect to the generators of the Laguerre tessellation can be read off from, e.g.,
\cite{BirginEtAl2021,BourneRoper2015,deGournay2018,Kitagawa_Merigot_Thibert_2019}.
%We give expressions for the second derivatives of $H$ in the appendix. 
The assumption that $H$ is 3-times continuously differentiable in a neighbourhood of $\bfX^*$ is `generically' true but may fail to hold if the diagram $\{L_i(\bfX^*;\bfv)\}_{i=1}^n$ is `degenerate' in the sense that there exists adjacent cells that intersect in a set of $\mathcal{H}^{d-1}$-measure zero (e.g., for $d=2$, if four cells meet at a point, such as in a checkerboard configuration). In this case the combinatorics of some of the Laguerre cells (the number of faces, edges, vertices, etc.) may change in a neighbourhood of $\bfX^*$, and $H$ may fail to be 3-times continuously differentiable.

We believe that the assumption $\mathrm{dim}(\mathrm{ker}(D^2 H(\bfX^*)))=1+d$ is also `generically' true, even if it is not always true. Due to the translation- and dilation-invariance of $\nabla H$,  we have $\mathrm{dim}(\mathrm{ker}(D^2 H(\bfX))) \ge 1+d$ for all $\bfX \in \distinct_n$.
To see this, differentiate the expression $\nabla H(\lambda \bfX)=\nabla H(\bfX)$ with respect to $\lambda$ and set $\lambda=1$ to obtain $D^2 H(\bfX) \bfX = 0$. Therefore $\bfX \in \mathrm{ker}(D^2 H(\bfX))$. Similarly, let $e_i \in \mathbb{R}^d$ denote the $i$-th standard basis vector. Then differentiating $\nabla H(\bfX+s(e_i,\ldots,e_i))=\nabla H(\bfX)$ with respect to $s$ and setting $s=0$ gives $(e_i,\ldots,e_i) \in \mathrm{ker}(D^2 H(\bfX))$ for all $i \in \{1,\ldots,d\}$.
Numerical evidence suggests that typically $\mathrm{dim}(\mathrm{ker}(D^2 H(\bfX)))=1+d$, but it is possible to construct examples where $\mathrm{dim}(\mathrm{ker}(D^2 H(\bfX))) > 1+d$. For example, if a Laguerre cell $L_i$ on the boundary of $\Omega$ only has one neighbour $L_j$  (e.g., if $L_i$ is a triangular Laguerre cell in the corner of a square domain), then the seed $x_i$ and the weight $w_i$ can be adjusted appropriately ($x_i \mapsto x_i + s(x_i-x_j)$, $s>0$), 
while keeping the other seeds and weights fixed,
without changing the Laguerre diagram, and hence without changing $\nabla H$. This corresponds to another vector in the nullspace of $D^2 H$.
\end{remark}

\begin{proof}[Proof of Theorem \ref{thm:main}]
Let $\bfX^*=(x_1^*,\ldots,x_n^*) \in B \cap \distinct_n$ be the global maximiser of $H$ 
 given in the statement of Theorem \ref{thm:main}. Note that $\bfX^*$ satisfies $|\bfX^*-\mathbf{T}|=R$ because of the following. If $|\bfX^*-\mathbf{T}|<R$, then $\nabla H(\bfX^*) = 0$ by the optimality of $\bfX^*$ in $B$. But this contradicts the assumption that there does not exists a Laguerre diagram compatible with $(\bfv,\bfB)$. 
 
%  Define 
% \[
% S = \{ \bfX=(x_1,\ldots,x_n) \in \distinct_n \, : \, x_1=x_1^*, \; |\bfX - \mathbf{T}| = R\}.
% \]
% Then $\bfX^* \in S$.
% Recall that $f=|\nabla H|^2$.
% By equation \eqref{eq:invariance grad H}, $\bfX^*$ is a local minimiser of $f$ on $\distinct_n$ if and only if it is a local minimiser of $f$ on $S$.

 The proof is split into three parts. Firstly, we write down the KKT conditions (first-order optimality conditions) for the optimisation problem $\max_B H$, which are satisfied by $\bfX^*$. Secondly, we use these conditions to show that $\nabla f(\bfX^*)=0$. 
 % the KKT conditions for the  optimisation problem $\min_{S} f$. 
 Finally, we show that $\bfX^*$ 
 is a local minimiser of $f$ by checking the second-order, sufficient optimality conditions. %for the optimisation problem %$\min_{S} f$.   
 
  Let $c_{\mathrm{ball}}:\mathbb{R}^{nd} \to \mathbb{R}$ be the constraint function given by $c_{\mathrm{ball}}(\bfX)=R^2-|\bfX-\mathbf{T}|^2$. Then $B =\{ \bfX \in \mathbb{R}^{nd} : c_{\mathrm{ball}}(\bfX) \ge 0 \}$.
       By assumption, $\bfX^*$ is a minimiser of $-H$ on $B$. The constraint $c_{\mathrm{ball}} \ge 0$ is active at $\bfX^*$, which means that $c_{\mathrm{ball}}(\bfX^*)=0$. Observe that 
       \[
\nabla c_{\mathrm{ball}}(\bfX^*) = -2 (\bfX^*-\mathbf{T}) \ne 0
       \]
       because $\bfX^* \in \distinct_n$.
Therefore $\bfX^*$ satisfies the linear independence constraint qualification (LICQ) (see \cite[Definition 12.4]{NocedalWright}) and so, by the Karush-Kuhn-Tucker Theorem \cite[Theorem 12.1]{NocedalWright}, there exists a Lagrange multiplier $\lambda^* \ge 0$ such that
\begin{equation}
\label{eq:KKT}    
- \nabla H(\bfX^*) = \lambda^* \nabla c_{\mathrm{ball}}(\bfX^*) = - 2 \lambda^* (\bfX^* - \mathbf{T}). 
\end{equation}
% Next we show that $\bfX^*$ satisfies the KKT conditions for the constrained optimisation problem $\min_{S} f$. For $i \in \{1,\ldots,d\}$, define
% $c_i : \mathbb{R}^{nd} \to \mathbb{R}$ by $c_i(\bfX)=(x_1 - x_1^*) \cdot e_i$, where $e_i \in \mathbb{R}^d$ denotes the $i$-th standard basis vector in $\mathbb{R}^d$. Then 
% \[
% S = \{\bfX \in \distinct_n \, : \, c_{\mathrm{ball}}(\bfX) =0, \; c_i(\bfX) = 0 \; \forall \; i \in \{1,\ldots,d\} \}.
% \]
% Let $\mathcal{L}: \distinct_n \times \mathbb{R}^{d+1} \to \mathbb{R}$ be the Lagrangian for the constrained optimisation problem $\min_{S} f$:
% \[
% \mathcal{L}(\bfX,\mu) = 
% f(\bfX) - \sum_{i=1}^d \mu_i c_i(\bfX) - \mu_{d+1} c_{\mathrm{ball}}(\bfX).
% \]
% Let $\mu^*=0$. 

Next we show that $\bfX^*$ is a critical point of $f$. By equation \eqref{eq:KKT},
\[
%\nabla_{\bfX} \mathcal{L}(\bfX^*,\mu^*) 
\nabla f(\bfX^*)
= 2 D^2 H(\bfX^*) \nabla H(\bfX^*) = 4 \lambda^* D^2 H(\bfX^*)(\bfX^* - \mathbf{T}) = 0
\]
since $\bfX^*$ and $\mathbf{T}$ belong to $\mathrm{ker}(D^2 H(\bfX^*))$; see Remark \ref{remark:assumptions}. 
%Therefore $(\bfX^*,\mu^*)$ satisfies the KKT conditions for  the constrained optimisation problem $\min_{S} f$.
%Therefore $\bfX^*$ is a critical point of $f$.

Finally, we show that 
$\bfX^*$ is a local minimiser of $f$.
%Finally, we check the second-order sufficient optimality conditions.
To compute $D^2 f$ we will use the Einstein summation convention (sum over repeated indices), with subscript Roman indices corresponding to seeds and
 superscript Greek indices corresponding to components in $\mathbb{R}^d$. 
For example, $x_i^\alpha$ denotes component $\alpha \in \{1,\ldots,d\}$ of seed $i \in \{1,\ldots,n\}$. For all $i,j \in \{1,\ldots,n\}$ and $\alpha,\beta \in \{1,\ldots,d\}$,
\begin{align}
\nonumber
\frac{\partial f}{\partial x_i^\alpha} & = 
2 \, \frac{\partial^2 H}{\partial x_i^\alpha \partial x_k^\gamma} \,\frac{\partial H}{\partial x_k^\gamma},
\\
\label{eq:D2f}
\frac{\partial^2 f}{\partial x_i^\alpha \partial x_j^\beta} & =
2 \, \frac{\partial^3 H}{\partial x_j^\beta \partial x_i^\alpha \partial x_k^\gamma} \,\frac{\partial H}{\partial x_k^\gamma} + 
2 \, \frac{\partial^2 H}{\partial x_i^\alpha \partial x_k^\gamma} \,
\frac{\partial^2 H}{\partial x_j^\beta \partial x_k^\gamma},
\end{align}
where $k$ is summed over $\{1,\ldots,n\}$ and $\gamma$ is summed over $\{1,\ldots,d\}$. Recall that for all $\bfX \in \distinct_n$, $D^2 H(\bfX) \bfX = 0$, i.e., 
\[
\frac{\partial^2 H}{\partial x_i^\alpha \partial x_k^\gamma} x_k^\gamma = 0.
\]
Differentiating with respect to $x_j^\beta$ gives
\begin{equation}
\label{eq: Humza Yousaf}
\frac{\partial^3 H}{\partial x_j^\beta \partial x_i^\alpha \partial x_k^\gamma} \,x_k^\gamma +
\frac{\partial^2 H}{\partial x_i^\alpha \partial x_j^\beta} =0.
\end{equation}
Evaluating \eqref{eq:D2f} at $\bfX^*$ and using \eqref{eq:KKT} gives
\begin{align*}
\frac{\partial^2 f}{\partial x_i^\alpha \partial x_j^\beta} (\bfX^*) 
& = 
4 \lambda^* \frac{\partial^3 H}{\partial x_j^\beta \partial x_i^\alpha \partial x_k^\gamma} (\bfX^*) \, ((\bfX^*)_k^\gamma - (\mathbf{T})_k^\gamma) + 
2 \, \frac{\partial^2 H}{\partial x_i^\alpha \partial x_k^\gamma}(\bfX^*) \,
\frac{\partial^2 H}{\partial x_j^\beta \partial x_k^\gamma} (\bfX^*)
\\
&  = - 4 \lambda^* \frac{\partial^2 H}{\partial x_i^\alpha \partial x_j^\beta} (\bfX^*) 
 + 
2 \, \frac{\partial^2 H}{\partial x_i^\alpha \partial x_k^\gamma}(\bfX^*) \,
\frac{\partial^2 H}{\partial x_j^\beta \partial x_k^\gamma} (\bfX^*),
\end{align*}
where in the second line we used \eqref{eq: Humza Yousaf} and the fact that $\mathbf{T} \in \mathrm{ker}(D^2 H (\bfX^*)).$ Therefore
\begin{equation}
\label{eq: Nicola Sturgeon}
%D^2_{\bfX \bfX} \mathcal{L}(\bfX^*,\mu^*) = 
 D^2 f(\bfX^*) = 
- 4 \lambda^* D^2 H(\bfX^*) + 2 \left( D^2 H(\bfX^*) \right)^2.
\end{equation}
Recall that $\lambda^* \ge 0$ and $D^2 H(\bfX^*)$ is negative semi-definite (since $H$ is concave). Therefore $D^2 f(\bfX^*)$ is positive semi-definite. Moreover, 
if $\bfY \in \mathbb{R}^{nd}$ satisfies $\bfY \cdot D^2 f(\bfX^*) \bfY = 0$, then $\bfY \in \mathrm{ker}(D^2 H(\bfX^*))$, namely $\bfY = a \bfX^* + (b,\ldots,b)$ for some $a \in \mathbb{R}$, $b \in \mathbb{R}^d$; see Remark \ref{remark:assumptions}. But 
\[
f(\bfX^* + \bfY) = 
f((1+a)\bfX^* + (b,\ldots,b)) = f(\bfX^*)
\]
by Theorem \ref{thm:properties H} \ref{Thrm4.2AffineTrans}. Therefore
$\bfZ \cdot D^2 f(\bfX^*) \bfZ > 0$ for all directions $\bfZ \in \mathbb{R}^{nd}$ except for those directions along which $f$ is constant.
%\mrp{I think ending here is fine.}
%\db{[DB: Do we need more details here? I don't think so but, just in case, here they are:

Let $\bfZ \in \mathbb{R}^{nd}$ and write $\bfZ = \bfY + \bfY^\perp$, where $\bfY \in \mathrm{ker}(D^2 H(\bfX^*))$ and $\bfY^\perp$ belongs to the orthogonal subspace of $\mathbb{R}^{nd}$. We will show that $f(\bfX^* + \bfZ) \ge f(\bfX^*)$ if $|\bfZ|$ is sufficiently small. We can assume that $\bfY^\perp \ne 0$, otherwise $f(\bfX^* + \bfZ) = f(\bfX^*)$. By Taylor's Theorem there exists $t \in (0,1)$ such that
\[
f(\bfX^* + \bfZ)
= f(\bfX^*) 
%+ \nabla f(\bfX^*) \cdot \bfZ 
+ 
\frac 12 \bfZ \cdot D^2 f(\bfX^* + t \bfZ) \bfZ.
\]
Note that 
\[
\bfZ \cdot D^2 f(\bfX^*) \bfZ
= \bfY^\perp \cdot D^2  f(\bfX^*) \bfY^\perp > 0
\]
since $\bfY^\perp \ne 0$. Since $D^2 f$ is continuous, it follows that $\bfZ \cdot D^2 f(\bfX) \bfZ > 0$ for all $\bfX$ in a sufficiently small neighbourhood of $\bfX^*$. Therefore $f(\bfX^*+ \bfZ) > f(\bfX^*)$ if $|\bfZ|$ is sufficiently small. It follows that $\bfX^*$ is a local minimiser of $f$, as required.
\end{proof}
% We show that it is in fact positive definite on the set of linearised feasible directions 
% \begin{align*}
% \mathcal{F}(\bfX^*) 
% & =
% \{ \mathbf{W} \in \mathbb{R}^{nd} \, : \, 
% \mathbf{W} \cdot \nabla c_{\mathrm{ball}} (\bfX^*) = 0, \; 
% \mathbf{W} \cdot \nabla c_i (\bfX^*) = 0 \; \forall \; i \in \{1,\ldots,d\} \} 
% \\
% & = \{ \mathbf{W}=(w_1,\ldots,w_n) \in \mathbb{R}^{nd} \, : \, \mathbf{W} \cdot (\mathbf{X}^* - \mathbf{T}) = 0, \; w_1 = 0 \}.
% \end{align*}
% Suppose that $\mathbf{W} \in \mathcal{F}(\bfX^*)$ satisfies 
% $\mathbf{W} \cdot D^2_{\bfX \bfX} \mathcal{L}(\bfX^*,\mu^*) \mathbf{W} =0$. Then $\mathbf{W} \in \mathrm{ker}(D^2 H(\bfX^*))$ by equation \eqref{eq: Nicola Sturgeon}. Therefore there exists $a \in \mathbb{R}$ and $b \in \mathbb{R}^d$ such that
% \[
% \mathbf{W} = a \bfX^* + (b,\ldots,b)
% \]
% by the characterisation of the kernel of $D^2 H(\bfX)$; see Remark \ref{remark:assumptions}. But $w_1 =0$, which implies that $b = - a x_1^*$. Therefore $\mathbf{W} = a (\bfX^* - (x_1^*,\ldots,x_1^*))$. ........... \db{[DB: Problem: $\mathbf{W} \cdot (\bfX^* - \mathbf{T}) = 0$ does not imply that $a=0$ unless $t=x_1^*$. Can we modify the definition of $S$ or $B$? Maybe include the constraint $x_1=t$ in $B$? But this will have consequences for \eqref{eq:KKT}. At least the proof above shows that $\bfX^*$ is a critical point of $f$ in $\distinct_n$. (We don't need the second-order condition for this.)]}

\begin{remark}[Application of Theorem \ref{thm:main}: Justification of the heuristic from \cite{TefraRowenhurtsDirectParameterEstimateforGPBD}]
\label{remark:heuristic}
A popular heuristic for fitting Laguerre diagrams to volume and centroid data (approximately solving $\min f$) is to take the seeds to be the target centroids, $\bfX=\bfB$; see for example \cite{TefraRowenhurtsDirectParameterEstimateforGPBD}. We give a partial justification for this heuristic. We can decompose $H$ as
\[
H(\bfX) = F(\bfX) + G(\bfX), \quad \textrm{where} \quad G(\bfX) = - \frac{1}{2}\sum_{i=1}^{n} v_i |x_i|^2 + \sum_{i=1}^n v_i x_i \cdot b_i - \frac{1}{2}\int_{\Omega} |x|^2 \, \text{d}x
\]
and $F$ was defined in \eqref{eq:F}.
Consider maximising $H$ over
$B=\{\bfX \in \mathbb{R}^{nd} : |\bfX - \mathbf{T} | \le R\}$, 
where $\mathbf{T}=(\sigma(\Omega),\ldots,\sigma(\Omega))$ and $\sigma(\Omega)$ is the centroid of $\Omega$.
%the compact set $\Omega^n$ 
This is equivalent to locally minimising the fitting error $f$ by Theorem \ref{thm:main}. The maximiser of $H$ depends on the competition between the terms $F$ and $G$. It can be shown that the first term, $F$, is maximised when all the seeds coincide and are as far from the centroid of $\Omega$ as possible, namely when $\bfX=(x_1,\ldots,x_1) \in \partial B$. This is because maximising $F$ is equivalent to finding the worst approximation of the Lebesgue measure by a discrete measure in the Wasserstein metric.
%are in the same place and lie
%equal and lie at the point in 
%$\Omega$ that is furthest from the centroid of $\Omega$. 
On the other hand, the second term, $G$, is globally maximised over $\mathbb{R}^{nd}$ by $\bfX = \bfB$, namely when the seeds are the target centroids. This is the heuristic from \cite{TefraRowenhurtsDirectParameterEstimateforGPBD}. If the second term dominates, then we would expect the maximiser of $H$ (and the minimiser of $f$) to be close to $\bfB$. This partially justifies the heuristic used by \cite{TefraRowenhurtsDirectParameterEstimateforGPBD} and other authors. On the other hand, if the first term dominates, then the maximiser of $H$ may not belong to $\distinct_n$. This partially explains the numerical observations given in Section \ref{subsec: larger perts}.
\end{remark}

\begin{remark}[The set $K$ should be a ball]
\label{remark:Counterexample to the Bourne conjecture}
In Theorem \ref{thm:main} it is important that the set $B$ where we maximise $H$ is a ball. By Theorem \ref{thm:properties H} \ref{Thrm4.2AffineTrans}, if the infimum of $f$ is attained,
we have $\min_{\distinct_n} f = \min_{K \cap \distinct_n} f$ for any compact set $K$ containing a neighbourhood of a point of the form $\mathbf{T}=(t,\ldots,t) \in \mathbb{R}^{nd}$, not just for balls. However, we cannot replace the ball $B$ by an arbitrary compact set $K$ in the statement of Theorem \ref{thm:main}. The fact that $H$ is 1-positively homogeneous is also important. In general, it is not true that the maximiser of a concave function on a compact set is a minimiser of the norm of its gradient. For example, let $g:\mathbb{R}^2 \to \mathbb{R}$ be the concave function $g(x)=-\frac{1}{2} (x_1-2)^2-2(x_2-1)^2$. Let $c:\mathbb{R}^2 \to \mathbb{R}$ be the constraint function $c(x)=1-\left(\frac{x_1}{2}\right)^2-x_2^2$ and $S$ be the ellipse $S=\{ x \in \mathbb{R}^2 : c(x) \ge 0\}$. It can be shown that the global maximiser of $g$ on $S$ is $x^*=(\sqrt{2},1/\sqrt{2}) \in \partial S$, but the global minimiser of the convex function $|\nabla g|^2$ on $S$ is $y^*=(2y_2/(4-3y_2),y_2) \in \partial S$ with $y_2= \tfrac 13 (2-\sqrt{5}+\sqrt{3+2\sqrt{5}})\approx (1.108097,0.832484) $. Note that $y^* \ne x^*$. 
\end{remark}

%\db{[DB: Should we include the following or cut it?]}
We include the following theorem since it may be of general interest in convex optimisation. It gives conditions under which minimising a 1-positively homogeneous convex function on a ball is equivalent to locally minimising the norm of its gradient. We do not include its proof since it is very similar to the proof of Theorem \ref{thm:main}.

 \begin{proposition}[Minimising the norm of the gradient of a convex function]
 \label{thm:min g = min |grad g|}
Let $g: \mathbb{R}^{n} \to \mathbb{R}$ be convex.
Suppose that there exists $x_0 \in \mathbb{R}^n$ such that, for all $\lambda >0$, $x \in \mathbb{R}^n$,
\begin{equation}
\label{eq:1-positively Roper homogeneous}    
g(\lambda (x-x_0) + x_0) = \lambda g(x).
\end{equation}
(If $x_0=0$, then $g$ is 1-positively homogeneous.)
Assume that $g$ is continuously differentiable on $\mathbb{R}^n \setminus \{x_0\}$.
Let $R>0$ and
$B_R=\{x \in \mathbb{R}^n : |x - x_0 | \le R\}$. 
Assume that the global minimum of $g$ on $B_R$ is achieved at a point $x^* \in \partial B_R$.
Moreover, assume that $g$ is 3-times continuously differentiable in a neighbourhood of $x^*$ and $\mathrm{ker}(D^2 g(x^*))=\mathrm{span}_\mathbb{R} \{x^*-x_0 \}$. 
Then $x^*$ is a local minimiser of $|\nabla g|$ on  
$\mathbb{R}^n \setminus \{ x_0 \}$.
\end{proposition}

\section{Numerical experiments}
\label{sec: numerics}
In this section we provide numerical experiments to illustrate the theory from Section \ref{sec:main results}.

\subsection{Recovering a Laguerre diagram from the areas and centroids of its cells}
\label{sec: numerics - diagram exists}
First we show how to 
achieve Goal \ref{goal: recovering} using Theorem \ref{thm:properties H}. In particular, we reconstruct a 2D Laguerre diagram given the areas and centroids of its cells by maximising $H$. 

We take $\Omega=[0,1]^2$, $n=20$, we draw the seeds $\bfX_0 \in \Omega^n$ at random from the uniform distribution on $\Omega$, and we set all the weights to be zero, $\bfw_0 = (0,\ldots,0)$, so that $(\bf{X}_0,\bf{w}_0)$ generates a Voronoi tessellation of $\Omega$. For $i \in \{1,\ldots,n\}$, we define $v_i$ to be the area of $\Lag_i(\bfX_0,\bfw_0)$ and $b_i$ to be the centroid of $\Lag_i(\bfX_0,\bfw_0)$.
By Theorem \ref{thm:properties H} \ref{Thrm4.2:HisC1}, we can recover the Voronoi diagram $\{ \Lag_i(\bfX_0,\bfw_0) \}_{i=1}^n$ from $\vB$ by numerically maximising the concave function $H$.

For numerical robustness, since Laguerre diagrams are only defined if the seeds are distinct, we found that it was necessary to maximise $H$ subject to the constraint that the pairwise distance between the seeds is not too small, namely, that $c_{ij}(\bfX) \ge 0$ for all $i,j \in \{1,\ldots n\}$, $i<j$, where $c_{ij}: \mathbb{R}^{nd} \to \mathbb{R}$ are defined by
\[
c_{ij}(\bfX) = |x_i-x_j|^2 - \delta^2,
\]
where we take $\delta=10^{-3}$ in all the experiments below. Without this constraint, an iterative numerical optimisation algorithm may produce an iterate with seeds very close together.
%\sr{\sout{even though the maximiser $\bfX_0$ of $H$ has distinct seeds.}} 
As a very simple example, consider two seeds in one-dimension. If the initial guess %\sr{\sout{for the maximiser of $H$}} 
has the seeds in the wrong order, then any convergent numerical optimisation algorithm would reorder the seeds, which may cause the seeds to get very close together as they swap positions. This could give rise to some numerical instability since it is difficult to compute Laguerre diagrams and solve optimal transport problems when the seeds are very close together. We have observed the analogue of this behaviour in two dimensions.

In addition, we impose the constraint $c_{\textrm{ball}}(\bfX) \ge 0$, where
$c_{\textrm{ball}}: \mathbb{R}^{nd} \to \mathbb{R}$ is defined by
\[
c_{\textrm{ball}}(\bfX) = R^2 - \sum_{i=1}^n |x_i-\sigma(\Omega)|^2, \qquad R = \sqrt{\text{area}(\Omega) \cdot n},
\]
and where $\sigma(\Omega)$ is the centroid of $\Omega$.
This constraint ensures that $\bfX$ lies in the ball in $\mathbb{R}^{nd}$ of radius $R$ and centre $(\sigma(\Omega),\ldots,\sigma(\Omega))$. In particular, it reduces the domain of the optimisation problem to a compact set. This can be done without loss of generality by Lemma \ref{Lemma:Meyron}.
% (Note, however, that it does affect the maximum value of $H$ in the case where the data $\vB$ is not generated from a Laguerre diagram, as in Section \ref{sec:fitting}.)

In summary, we recover $\{ \Lag_i(\bfX_0,\bfw_0) \}_{i=1}^n$ from $\vB$ by numerically solving the following constrained optimisation problem:
\begin{equation}
\label{eq:maxHcon}    
\max \left\{ H(\bfX;\vB) : \bfX \in \mathbb{R}^{nd}, \; c_{\textrm{ball}}(\bfX) \ge 0, \; c_{ij}(\bfX) \ge 0 \; \forall \; i,j \in \{1,\ldots,n\}, \, i<j \right\}.
\end{equation}
%\sr{Chek whether any normalisation was used.}
Note that the maximum value is $0$ by Theorem \ref{thm:properties H} \ref{Thrm4.2C=BthenGradH=0}. As described above, we added the constraints $c_{ij} \ge 0$ for numerical stability. The price to pay is that the optimisation problem is no longer convex (the constraint set is not convex). An alternative approach, to preserve convexity, would be to maximise $H$ over the convex set $\{ c_{\mathrm{ball}} \ge 0\}$ using a method for non-smooth convex optimisation (such as a proximal method) that does not need $\nabla H$ to be well-defined everywhere (note that $H$ can be evaluated in a robust way even if the seeds are not distinct).

We solved the optimisation problem \eqref{eq:maxHcon} in Python.
To compute Laguerre diagrams and solve the semi-discrete optimal transport problem we used the Python library \emph{pysdot} \cite{pysdot} (to compute $\bfw^*(\bfX;\bfv)$). The initial guess for the optimal  
transport solver was generated using the rescaling method from \cite[Section 2.2]{Meyron2019}, which is based on Lemma \ref{Lemma:Meyron}. The optimal transport algorithm (the damped Newton method \cite{Kitagawa_Merigot_Thibert_2019}) was terminated when the percentage error of the areas of the Laguerre cells fell below $0.1\%$, namely, when $\bfw$ was such that
\[
100 \cdot \frac{|m_i(\bfX,\bfw)-v_i|}{v_i} < 0.1 \quad \forall \; i \in \{1,\ldots,n\}.
\]
We used the Python function \emph{scipy.optimize.minimize} to solve the constrained optimisation problem \ref{eq:maxHcon} with a random initial guess $\bfX_{\textrm{init}}$ (drawn from the uniform distribution on $\Omega^{nd}$), with the
optimisation method \emph{SLSQP}, and with the tolerance parameter \emph{ftol} equal to $10^{-10} |H(\bfX_{\textrm{init}})|$.
%\db{All of our code is available on GitHub.} 
%\db{[Will it be? Add link to code if we manage to upload this in time before submitting.]}

The results are shown in Figures \ref{fig:diagram exists - initial and final} and \ref{fig:diagram exists - F & H}. 
The optimisation algorithm terminated successfully after $229$ iterations.
Figure \ref{fig:diagram exists - initial and final}, left, shows the Laguerre tessellation $\{ L_i(\bfX_{\textrm{init}};\bfv)\}_{i=1}^n$ generated by the random initial guess $\bfX_{\textrm{init}}$, and Figure \ref{fig:diagram exists - initial and final}, right, shows the final Laguerre tessellation after $229$ iterations. The red dots are the actual centroids, which are almost indistinguishable from the blue dots, the target centroids $\bfB$. 
Figure \ref{fig:diagram exists - F & H} shows the convergence of $H$ and $F$ to $0$.
% The initial value of $H$ is $H(\bfX_{\textrm{init}})= \db{insert}$ and the final value of $H$ is $-1.0984 \times 10^{-8}$. 
% The ratio of the final value of $f$ to the initial value of $f$ is $0.29781 \times 10^{-9}$.

\begin{figure}
\centering
\includegraphics[width = 0.49\textwidth]{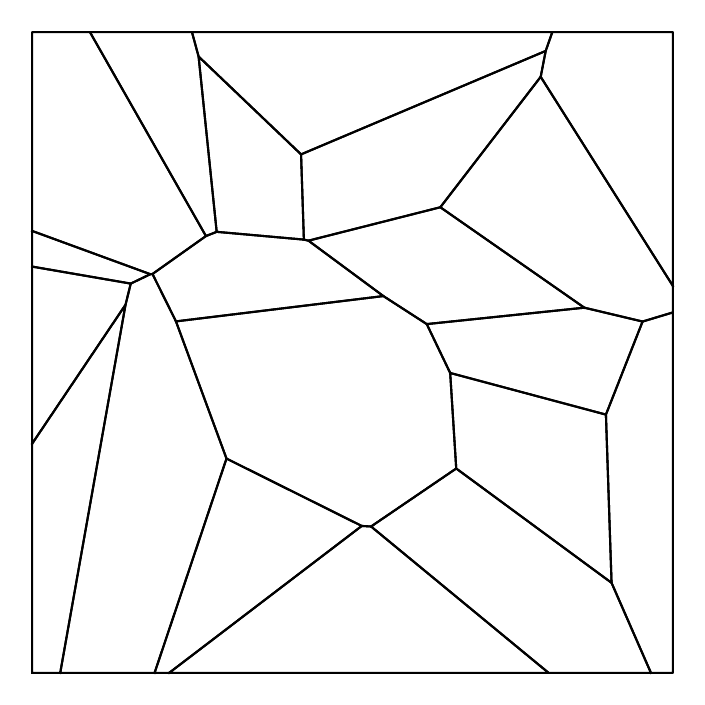}
\includegraphics[width = 0.49\textwidth]{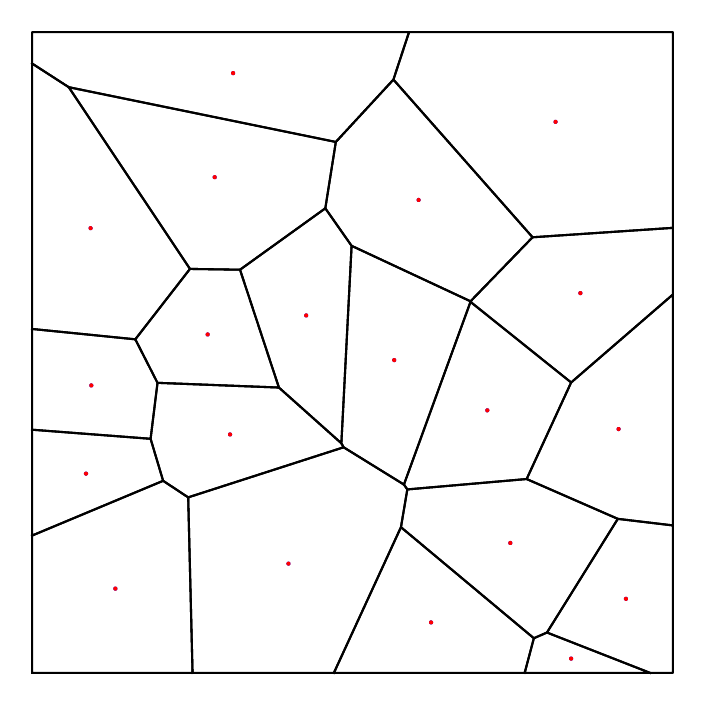}
\caption{Recovering a Laguerre diagram $\{ \Lag_i(\bfX_0,\bfw_0) \}_{i=1}^n$ from the areas $\bfv$ and centroids $\bfB$ of its cells; see Section \ref{sec: numerics - diagram exists}. Left: The Laguerre tessellation $\{ L_i(\bfX_{\textrm{init}};\bfv)\}_{i=1}^n$ with cells of areas $\bfv$ (up to a $0.1\%$ error) generated from a random collection of seeds  $\bfX_{\textrm{init}}$. Right: Numerical approximation of the unique Laguerre tessellation with cells of areas $\bfv$ and centroids $\bfB$. This figure shows the Laguerre tessellation corresponding to an approximate maximiser of the constrained optimisation problem \eqref{eq:maxHcon}, computed by \emph{scipy.optimize.minimize} using the initial guess $\bfX_{\textrm{init}}$ and $229$ iterations. The red dots are the centroids of the computed Laguerre cells. The blue dots (which are almost indistinguishable from the red dots) are the target centroids $\bfB$. The cells have areas $\bfv$ up to a $0.1\%$ error. We have not plotted the true Laguerre diagram $\{ \Lag_i(\bfX_0,\bfw_0) \}_{i=1}^n$ since it is indistinguishable from the computed Laguerre diagram.
}
\label{fig:diagram exists - initial and final}
\end{figure}

\begin{figure}
\centering
\includegraphics[width = 0.49\textwidth]{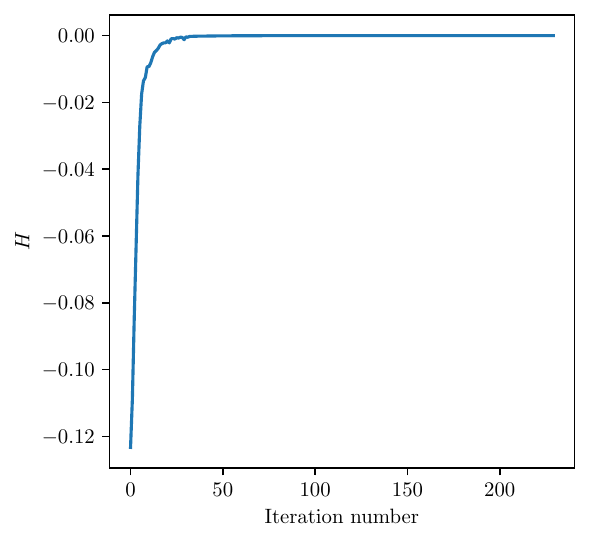}
\includegraphics[width = 0.49\textwidth]{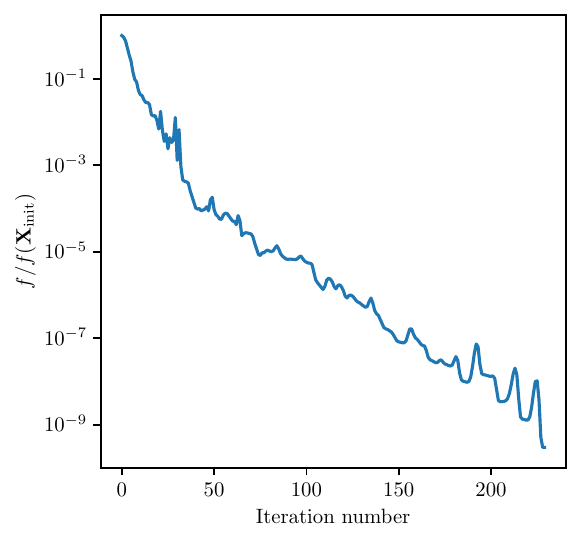}
\caption{Recovering a Laguerre diagram -- convergence of the algorithm from Section \ref{sec: numerics - diagram exists}. Left: The value of the objective function $H$ converges to its maximum value of $0$. 
%Note that $0$ is the maximum value because there exists a diagram with cells of areas and centroids $\vB$. 
Right: The nonlinear least squares error $f$ also converges to $0$. }
\label{fig:diagram exists - F & H}
\end{figure}

\subsection{Fitting a Laguerre diagram to synthetic data}
\label{sec:fitting}
In this section we consider Goal \ref{goal: fitting}. We fit a Laguerre tessellation to synthetic data 
%\sr{\sout{$\vB$}} 
that is obtained by perturbing a Voronoi diagram. We created the synthetic data as follows. First we took the Voronoi diagram $\{ \Lag_i(\bfX_0,\bfw_0) \}_{i=1}^n$ from Section \ref{sec: numerics - diagram exists} and computed the areas and centroids of its cells, $\vB$. Then we perturbed each centroid $b_i$ by a random vector \[
u_i^\varepsilon = \varepsilon r_i (\cos \theta_i,\sin \theta_i),
\]
where $r_i$ is drawn from the uniform distribution on $[0,1]$, $\theta_i$ is drawn from the uniform distribution on $[0,2 \pi]$, and $\varepsilon >0$ is the size of the perturbation.
Finally, so that the synthetic data is compatible (belongs to the set $\mathcal{D}_n$), we define
\[
b_i^\varepsilon = b_i + u_i^\varepsilon + \left( \sigma(\Omega) - \frac{1}{\mathrm{area}(\Omega)} \sum_{i=1}^n v_i (b_i + u_i^\varepsilon) \right),
\]
where $\sigma(\Omega)=(1/2,1/2)$ is the centroid of $\Omega=[0,1]^2$. Then $(\bfv,\bfB^\varepsilon) \in \mathcal{D}_n$ is our synthetic data. 

\subsubsection{Small perturbations}
\label{subsec: small}
First we considered a very small perturbation $\varepsilon = 0.001$. We fitted a Laguerre tessellation to the perturbed data $(\bfv,\bfB^\varepsilon)$ by solving the constrained optimisation problem \eqref{eq:maxHcon} as described above in Section \ref{sec: numerics - diagram exists}, with the same values of $\delta$, $R$ and \emph{ftol}, but with the initial guess $\bfX_{\mathrm{init}}=\bfB^\varepsilon$ (see Remark \ref{remark:heuristic} for a justification for using this initial guess). 

The algorithm terminated successfully after $402$ iterations. The results are shown in Figures \ref{fig:small perturb - initial and final}, 
\ref{fig:small perturb - H and F}
and \ref{fig:small perturb - pairwise dist}, left.
Figure \ref{fig:small perturb - initial and final}, left, shows the Laguerre tessellation $\{ L_i(\bfX_{\textrm{init}};\bfv)\}_{i=1}^n$ generated by the initial guess $\bfX_{\textrm{init}}=\bfB^\varepsilon$, and Figure \ref{fig:small perturb - initial and final}, right, shows the final Laguerre tessellation after $402$ iterations. The red dots are the actual centroids and the blue dots are the target centroids $\bfB^\varepsilon$. Since the perturbation is so small, the red and blue dots are almost indistinguishable. There are some minor visible differences between the Laguerre tessellation fitted to the unperturbed data (Figure \ref{fig:diagram exists - initial and final}, right) and the Laguerre tessellation fitted to the perturbed data (Figure \ref{fig:small perturb - initial and final}, right). 

As discussed in Remark \ref{remark:heuristic}, the initial guess $\bfX_{\mathrm{init}}=\bfB^\varepsilon$ is a good approximation of the maximiser of \eqref{eq:maxHcon}, at least in the eyeball metric -- compare Figure \ref{fig:small perturb - initial and final}, left, to Figure \ref{fig:small perturb - initial and final}, right. On the other hand, from Figure \ref{fig:small perturb - H and F} we see that the final value of least squares error $f$ is over three orders of magnitude smaller than $f(\bfX_{\mathrm{init}})$. The objective function $H$ increases from about $-5 \times 10^{-5}$ to $2.5 \times 10^{-4}$. Note that the maximum value of $H$ is positive because there does not exists a Laguerre diagram with cells of areas and centroids $(\bfv,\bfB^\varepsilon)$. Since $H$ can now take positive values, the constraint $c_{\mathrm{ball}} \ge 0$ in \eqref{eq:maxHcon} is necessary to ensure that the objective function is bounded by Theorem \ref{thm:properties H} \ref{Thrm4.2AffineTrans}.

Figure \ref{fig:small perturb - pairwise dist}, left, shows the minimum distance between the seeds at each iteration, normalised by $\delta$, which is the minimum distance between the seeds allowed by the constraint $c_{ij} \ge 0$. The dashed line corresponds to the constraint being active for some $i,j$. We see that the constraint becomes active around iteration $100$, as the seeds rearrange themselves, but by the end of the simulation the seeds are well separated, at least $30$ units of $\delta$ apart. This suggests that the maximum value of $H$ on the ball $\{ c_{\textrm{ball}} \ge 0\}$ is achieved at a point where all the seeds are distinct.

\begin{figure}
\centering
\includegraphics[width = 0.49\textwidth]{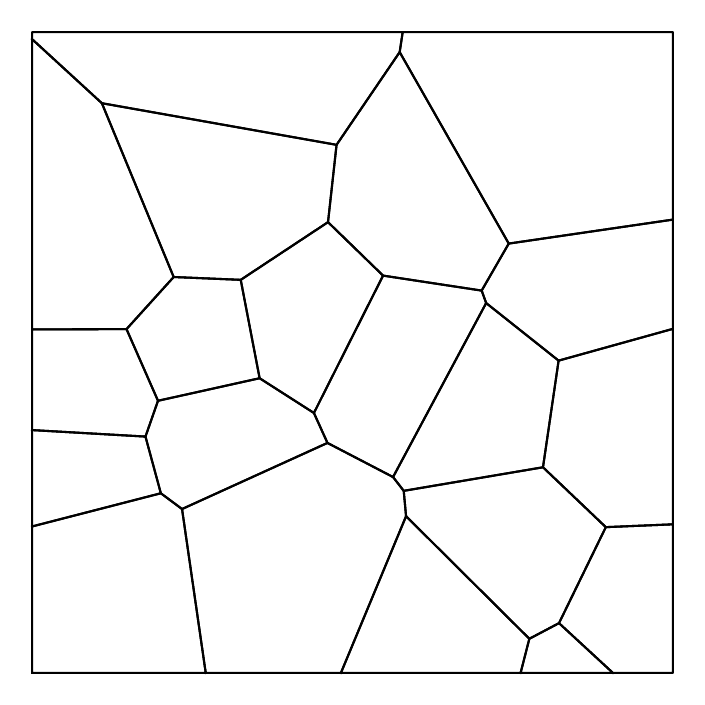}
\includegraphics[width = 0.49\textwidth]{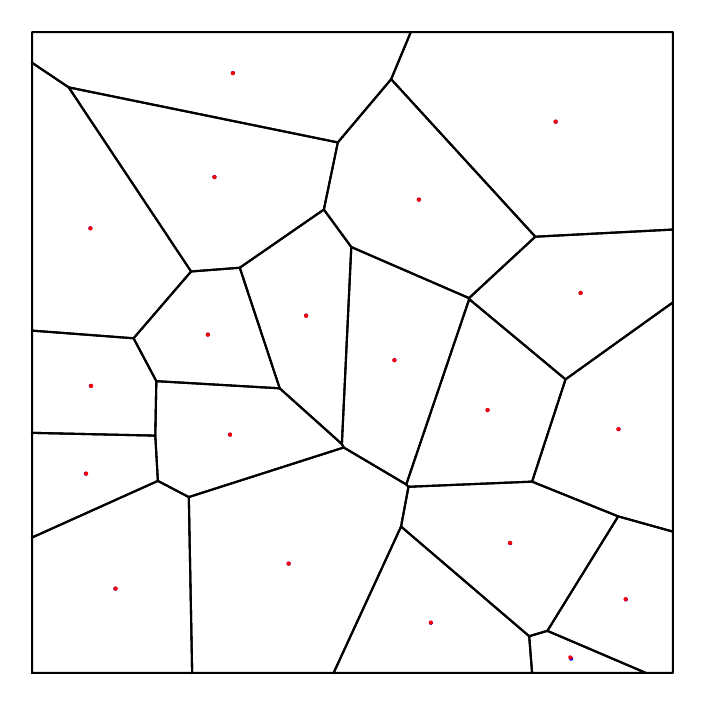}
\caption{Fitting a Laguerre diagram to synthetic data $(\bfv,\bfB^\varepsilon)$ with $\varepsilon=0.001$ (small perturbation); see Section \ref{subsec: small}. Left: The initial guess $\{ L_i(\bfX_{\textrm{init}};\bfv)\}_{i=1}^n$ for the maximiser of \eqref{eq:maxHcon}, where $\bfX_{\mathrm{init}}=\bfB^\varepsilon$.
Right: An approximate maximiser of the constrained optimisation problem \eqref{eq:maxHcon}, computed by \emph{scipy.optimize.minimize} using the initial guess $\bfX_{\textrm{init}}$ and $402$ iterations. The centroids of the Laguerre cells (red dots) are almost indistinguishable from the target centroids (blue dots).
}
\label{fig:small perturb - initial and final}
\end{figure}

\begin{figure}
\centering
\includegraphics[width = 0.49\textwidth]{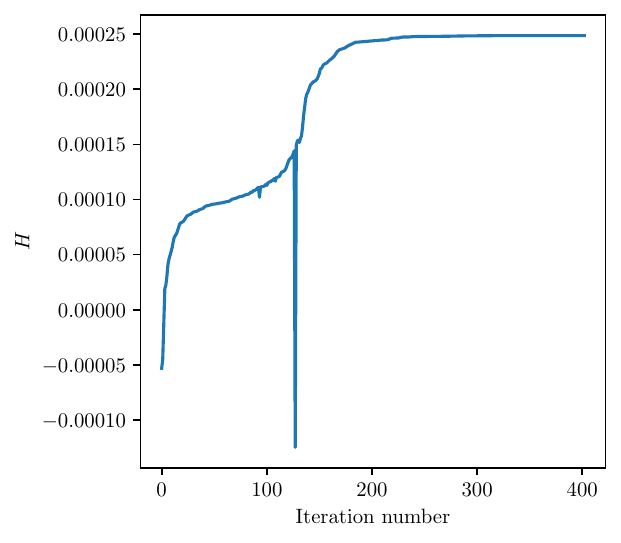}
\includegraphics[width = 0.49\textwidth]{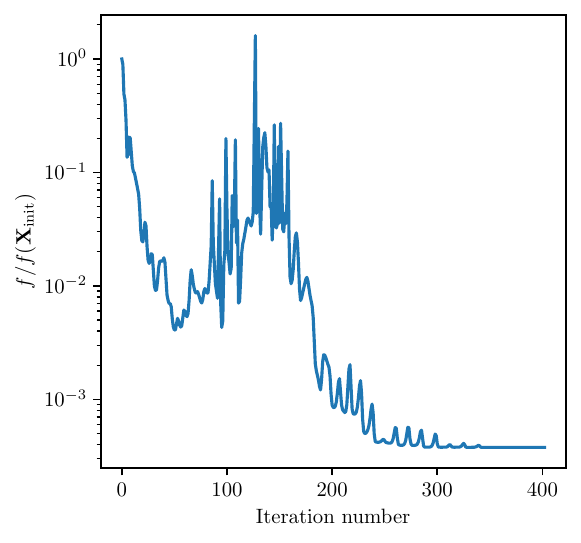}
\caption{Fitting a Laguerre diagram to synthetic data (small perturbation) - convergence of the algorithm from Section \ref{subsec: small}. Left: Convergence of the objective function $H$. It converges to a positive value because there does not exist a Laguerre diagram with cells of areas and centroids $(\bfv,\bfB^\varepsilon)$. Right: Convergence of the least squares error $f$. The $y$-axis shows the ratio of the value of $f$ at each iteration to its initial value $f(\bfX_{\mathrm{init}})$.}
\label{fig:small perturb - H and F}
\end{figure}

\begin{figure}[t]
\centering
\includegraphics[width = 0.49\textwidth]{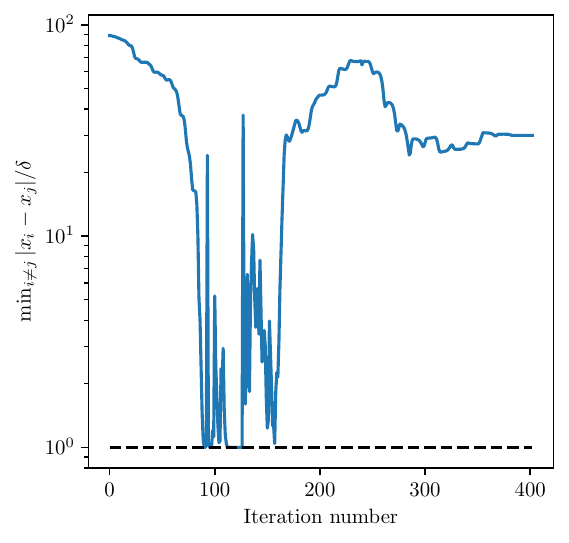}
\includegraphics[width = 0.49\textwidth]{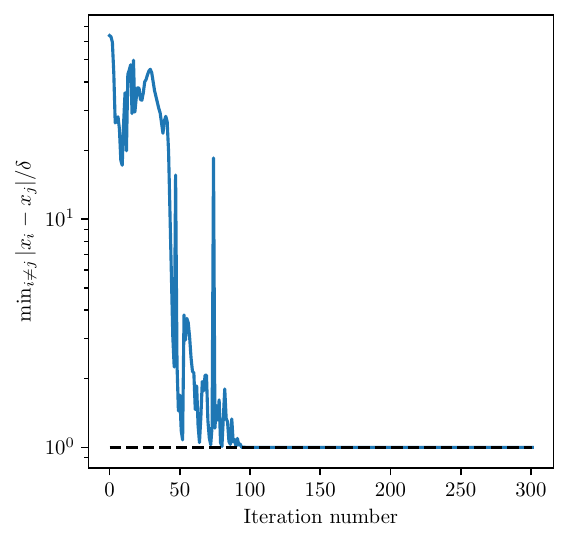}
\caption{Fitting a Laguerre diagram to synthetic data (see Section \ref{sec:fitting}) - minimum distance between the seeds for $\varepsilon=0.001$ (left) and $\varepsilon=0.05$ (right).
 The minimum distance between the seeds is normalised by the minimum allowed distance $\delta=10^{-3}$. The dotted line corresponds to when at least one of the constraints $c_{ij} \ge 0$ is active. Left (small perturbation): At the end of the algorithm none of the constraints are active. The minimum distance between the seeds is at least $30 \delta$. Right (larger perturbation): At the end of the algorithm at least one of the constraints is active.  
 This suggests that if we solved \eqref{eq:maxHcon} without the constraints $c_{ij} \ge 0$, then 
 some of the seeds would collide.}
\label{fig:small perturb - pairwise dist}
\end{figure}

\subsubsection{Larger perturbations}
\label{subsec: larger perts}

We repeated the simulations from the previous section with a slightly larger perturbation, $\varepsilon = 0.05$.
The algorithm terminated successfully after $301$ iterations and the results are shown in Figures \ref{fig:large perturb - initial and final}, 
\ref{fig:large perturb - H and F}
and \ref{fig:small perturb - pairwise dist}, right. 

Figure \ref{fig:large perturb - initial and final} shows the initial (left) and final (right) Laguerre tessellations, with the target centroids $\bfB^\varepsilon$ in blue and the centroids of the fitted diagram in red. This time there is a clear visible difference between the target and actual centroids, although the fit is still quite good. However, the fitted diagram (Figure \ref{fig:large perturb - initial and final}, right) is somewhat `irregular'; some of the cells are rather elongated and there is a tendency for some of the edges of the Laguerre cells to align. This suggests that the methods proposed in this paper of maximising $H$/minimising $f$ may not be entirely suitable for generating synthetic microstructures. It does, however, appear to work well for recovering Laguerre diagrams (Section \ref{sec: numerics - diagram exists}) and fitting Laguerre diagrams to EBSD data (Section \ref{sec:EBSD}). 

The convergence of $H$ and $F$ are illustrated in Figure \ref{fig:large perturb - H and F}. Figure \ref{fig:small perturb - pairwise dist}, right, shows that at least one of the constraints $c_{ij} \ge 0$ is active at the final iteration (and from iteration $100$ onwards). This suggests that, without the constraint $c_{ij} \ge 0$, some of the seeds would collide. In other words, it suggests that the maximum value of $H$ in the set $\{ c_{\textrm{ball}} \ge 0\}$ is achieved at a point where some of the seeds coincide. Therefore the assumptions of Theorem \ref{thm:main} are not satisfied, and there is no rigorous correspondence between maximising $H$ and finding local minimisers of the least squares error $f$. Nevertheless, Figure \ref{fig:large perturb - H and F} suggests that maximising $H$ still works reasonably well at minimising $f$ in practice. Moreover, we repeated this experiment by directly minimising $f$ (using the approach described in the following section, by solving \eqref{eq:minfcon}) and obtained very similar results. 
%\db{We do not present these results here due to lack of space, but they can be found on our GitHub page [add link].}

\begin{figure}[t]
\centering
\includegraphics[width = 0.49\textwidth]{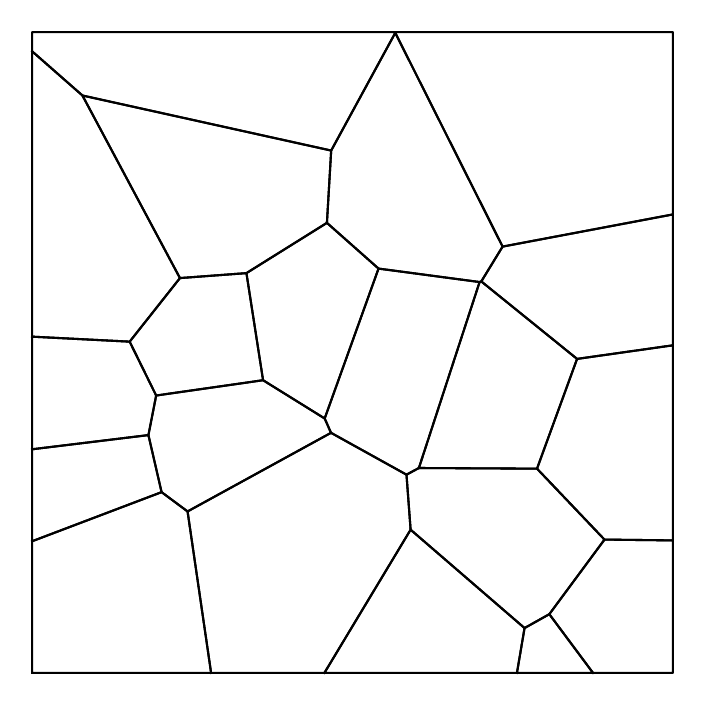}
\includegraphics[width = 0.49\textwidth]{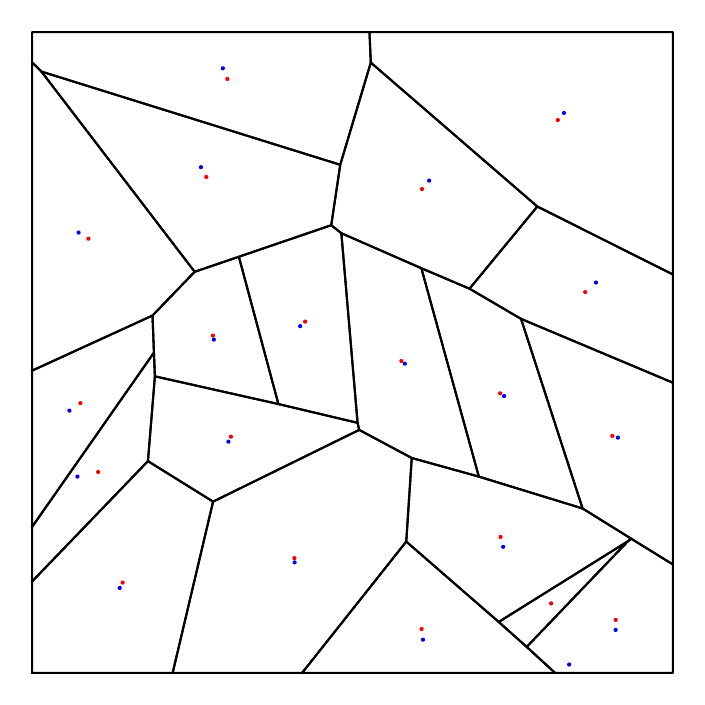}
\caption{Fitting a Laguerre diagram to synthetic data $(\bfv,\bfB^\varepsilon)$ with $\varepsilon=0.05$ (larger perturbation); see Section \ref{subsec: larger perts}.  Left: The initial guess $\{ L_i(\bfX_{\textrm{init}};\bfv)\}_{i=1}^n$ for the maximiser of \eqref{eq:maxHcon}, where $\bfX_{\mathrm{init}}=\bfB^\varepsilon$.
Right: An approximate maximiser of the constrained optimisation problem \eqref{eq:maxHcon}, computed by \emph{scipy.optimize.minimize} using the initial guess $\bfX_{\textrm{init}}=\bfB^\varepsilon$ and $301$ iterations. The centroids of the Laguerre cells are in red and the target centroids are in blue. The centroid error is relatively small, but the Laguerre cells are not very `regular'. In our simulations we found that fitting a diagram to synthetic data tends to produce some elongated cells or cells that are aligned (have almost parallel edges), as seen here.}
\label{fig:large perturb - initial and final}
\end{figure}

\begin{figure}
\centering
\includegraphics[width = 0.49\textwidth]{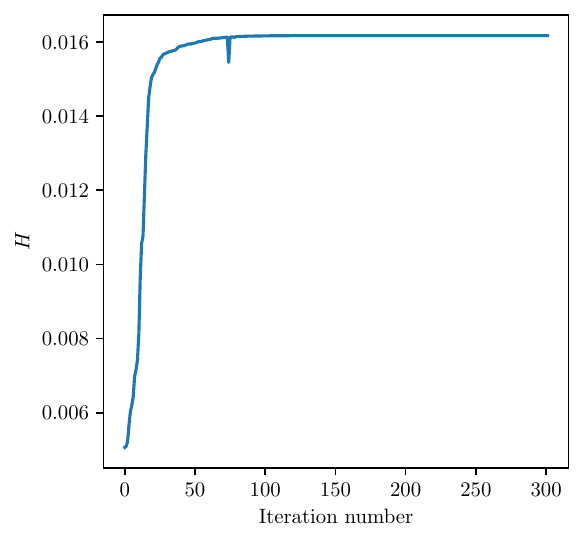}
\includegraphics[width = 0.49\textwidth]{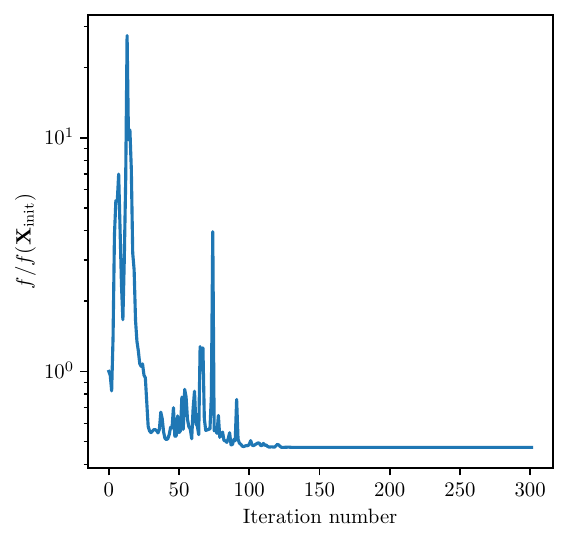}
\caption{Fitting a Laguerre diagram to synthetic data (larger perturbation) - convergence of the algorithm from Section \ref{subsec: larger perts}. Left: Convergence of the objective function $H$. Right: Convergence of the least squares error $f$. The $y$-axis shows the ratio of the value of $f$ at each iteration to its initial value $f(\bfX_{\mathrm{init}})$. We see that the initial guess $\bfX_{\mathrm{init}}=\bfB^\varepsilon$ is very good; the optimisation algorithm only decreases the least squares error $f$ by a factor of about $0.5$ (after initially increasing it significantly).}
\label{fig:large perturb - H and F}
\end{figure}

\subsection{Fitting a Laguerre diagram to EBSD data}
\label{sec:EBSD}

In this section we fit a Laguerre tessellation to an EBSD image provided by Tata Steel Research \& Development.
Figure \ref{fig:EBSD}, top, is an EBSD image of a single-phase steel. The pixels are coloured according to their crystallographic orientation. The orientation map is piecewise constant, and regions of constant orientation are known as \emph{grains}. There are $n=243$ grains in the image. From this image we extracted the areas $\bfv$ and the centroids $\bfB$ of the grains by counting the pixels in each grain.
%. \sr{[Mention how the volumes and centroids are calculated? By counting pixels.]} 
The domain is $\Omega=[0,252.25]^2$ (measured in microns).

To fit a Laguerre tessellation to the EBSD data $\vB$ we initially tried solving $\eqref{eq:maxHcon}$. However, as in Section \ref{subsec: larger perts}, our experiments suggested that the maximum value of $H$ on the ball $\{ c_{\mathrm{ball}} \ge 0 \}$ is achieved at a point where some of the seeds coincide. Since Theorem \ref{thm:main} does not apply in this case, we decided to directly minimise $f$, rather than maximise $H$. 
To be precise, 
we solved the  constrained optimisation problem 
\begin{equation}
\label{eq:minfcon}    
\min \left\{ \frac{n^2}{\mathrm{area}(\Omega)^3} \, f(\bfX;\vB) \, : \, \bfX \in \mathbb{R}^{nd}, \; c_{ij}(\bfX) \ge 0 \; \forall \; i,j \in \{1,\ldots,n\}, \, i<j \right\}.
\end{equation}
The normalisation factor $n^2/\mathrm{area}(\Omega)^3$ is simply included so that objective function is non-dimensional and scales roughly  constantly with respect to $n$ (this is based on the assumption the cells have roughly equal area).
The constraints $c_{ij} \ge 0$ are included as above for numerical robustness and to stop seeds colliding. 
Since $F$ is invariant under dilations and translations, we do not need to include the constraint $c_{\mathrm{ball}} \ge 0$.

We solved
\eqref{eq:minfcon} using \emph{scipy.optimize.minimize} with the initial guess $\bfX_{\textrm{init}}=\bfB$ and with the tolerance parameter \emph{ftol} equal to 
$10^{-8} n^2 f(\bfX_{\textrm{init}})/\mathrm{area}(\Omega)^3 = 1.6308 \times 10^{-10}$.

The results are shown in Figures \ref{fig:EBSD} and \ref{fig:EBSD - convergence}.
The optimisation algorithm terminated successfully after $714$ iterations.
Figure \ref{fig:EBSD}, bottom-right, shows the Laguerre cells together with the centroids $\bfB$ of the grains from the EBSD image in blue and the centroids of the fitted Laguerre cells in red. Figure \ref{fig:EBSD}, bottom-left, shows the fitted Laguerre diagram overlaid over the EBSD image. %Figure \ref{fig:EBSD}, bottom-right, shows the Laguerre diagram by itself. 
The quality of the fit is limited by the polygonal shape of the cells; in reality the grains may not be polygonal, as can be seen from the EBSD image in Figure \ref{fig:EBSD} (top).
Nevertheless, there is a clear resemblance, even though we are only minimising the least squares centroid error $f$, rather than the mismatch between the two figures (such as the number of misassigned pixels).
Figure \ref{fig:EBSD - convergence}, left, shows the convergence of the algorithm. The ratio of the initial value of $f$ to the final value of $f$ is $25.2$. Figure \ref{fig:EBSD - convergence}, right, shows the minimum distance between the seeds, normalised by $\delta$ (the minimum distance between the seeds allowed by the constraints $c_{ij} \ge 0$). This suggest that $\inf \{ f(\bfX;\vB) : \bfX \in \distinct_n \}$ is not attained (seeds collide along an infimising sequence), which is why we introduced the constraints $c_{ij} \ge 0$  in \eqref{eq:minfcon}.  

\begin{figure}
\centering
\includegraphics[width = 0.48\textwidth]{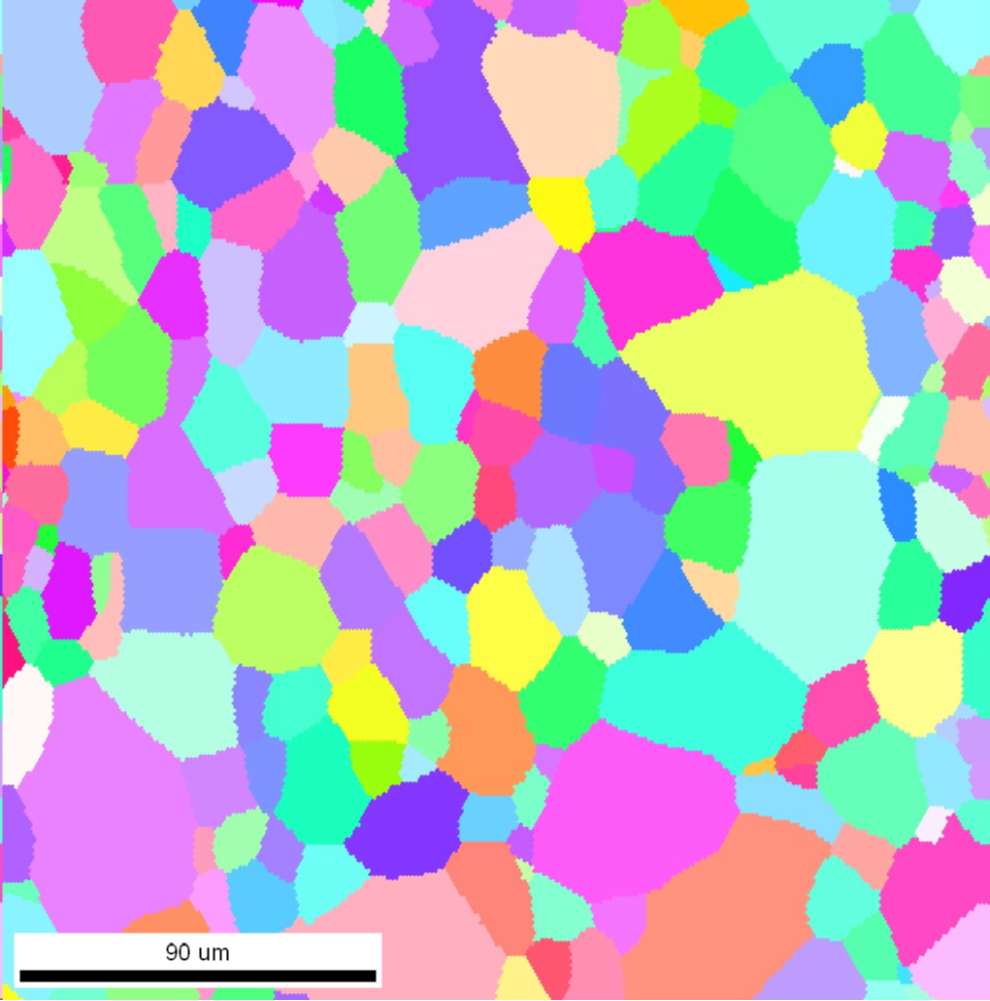}

\vspace{0.3cm}

\includegraphics[width = 0.48\textwidth]{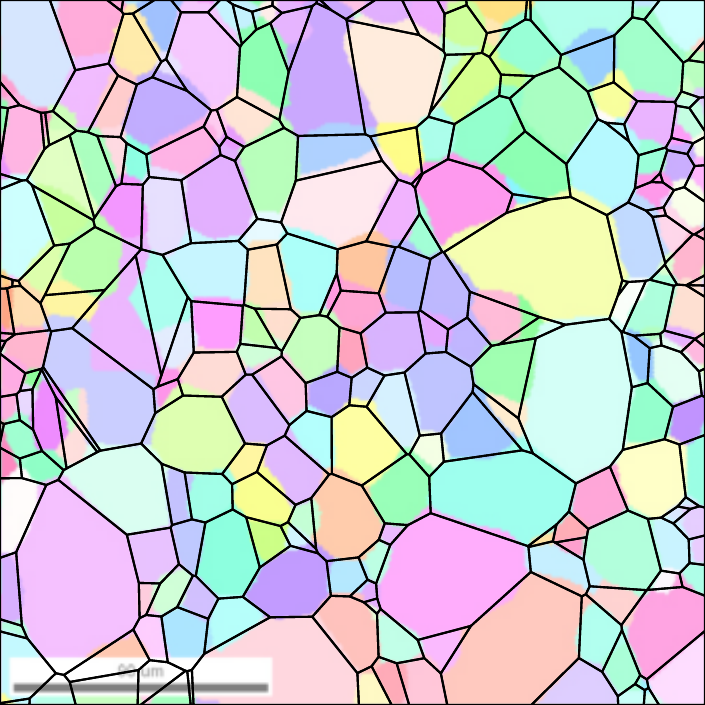}
\includegraphics[width = 0.48\textwidth]{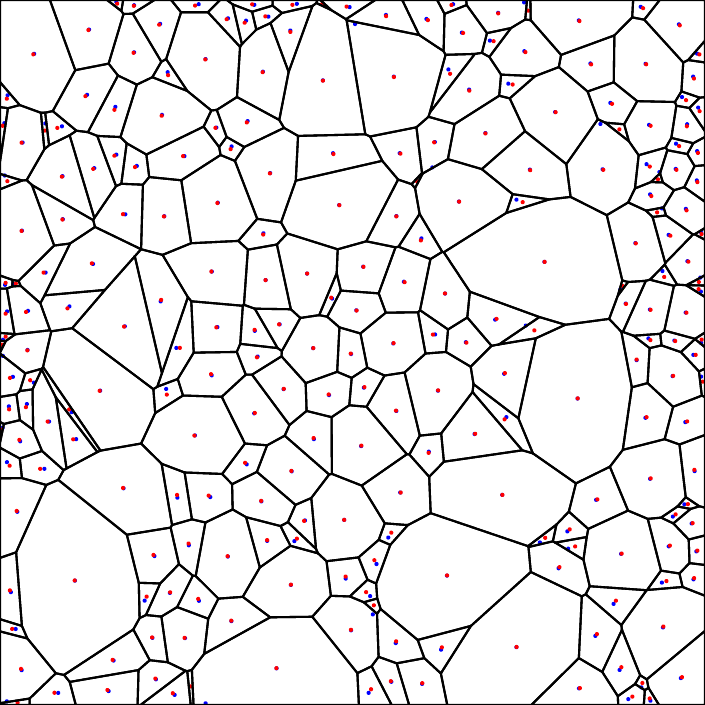}
\caption{Fitting a Laguerre tessellation to an EBSD image; see Section \ref{sec:EBSD}. Top: The original EBSD image of a single-phase steel. The grains are coloured according to the orientation of the crystal lattice. We extracted the areas and centroids $\vB$ of the grains from this EBSD image and then fitted a Laguerre tessellation to the image by solving the optimisation problem \eqref{eq:minfcon}. Bottom-left: The fitted Laguerre diagram overlaid on the EBSD image. Bottom-right: The centroids $\bfB$ from the EBSD data in blue and the centroids of the fitted Laguerre cells in red.}
\label{fig:EBSD}
\end{figure}

\begin{figure}[ht]
\centering
\includegraphics[width = 0.49\textwidth]{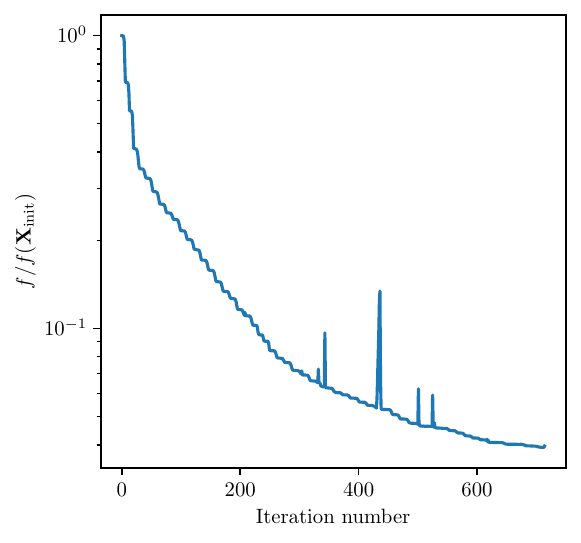}
\includegraphics[width = 0.49\textwidth]{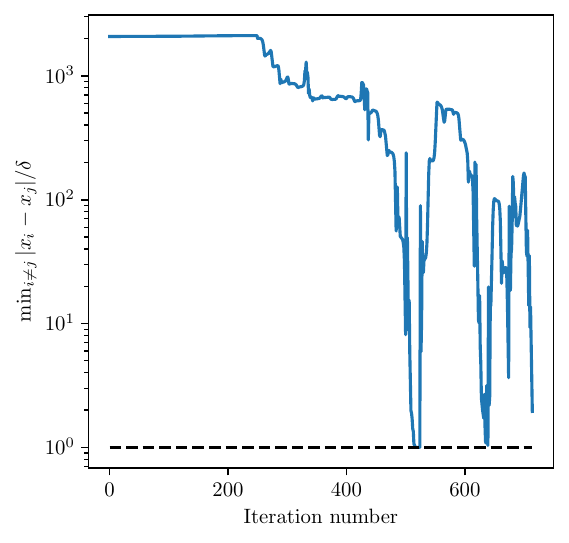}
\caption{Fitting a Laguerre diagram to an EBSD image - convergence of the algorithm from Section \ref{sec:EBSD}. Left: Decay of the objective function $f$, normalised by its initial value $f(\bfX_{\textrm{init}})$. The algorithm terminated after $714$ iterations and the objective function decreased by a factor of about $25$. Right: The minimum distance between the seeds, normalised by the minimum allowed distance $\delta=10^{-3}$. At least one of the constraints $c_{ij} \ge 0$ appears to be active for several iterations (the dotted line corresponds to an active constraint). This suggests that the seeds would collide in the absence of the constraints.}
\label{fig:EBSD - convergence}
\end{figure}

\section{Generalisation to anisotropic Laguerre tessellations}
    \label{sec:aniso}
    In this section we generalise some of our results to \emph{anisotropic Laguerre tessellations} (also known as \emph{anisotropic power diagrams} or \emph{generalised balanced power diagrams}), which have recently been used to model polycrystalline materials, e.g., \cite{altendorf20143d,ABGLP15,VBWBSJ16,TefraRowenhurtsDirectParameterEstimateforGPBD,PFWKS21,AFGK23,AFGK23B,jung2024analytical,BFRSB24}. 

    \begin{definition}[Anisotropic Laguerre tessellations] 
    Let $\mathbf{A} = (A_1,\ldots,A_n)$ be a tuple of $d$-dimensional, symmetric, positive-definite matrices $A_i \in \mathbb{R}^{d \times d}$. Let 
    $\bfX\in\distinct_n$
    %$\bfX \in \RR^{nd}$
    and $\bfw \in \RR^{n}$. 
    The $i^{\text{th}}$ \emph{anisotropic Laguerre cell} generated by $(\bfX,\bfw,\mathbf{A})$ is defined by
    \begin{equation}
        \Lag^{\mathbf{A}}_i(\bfX,\mathbf{w}) 
        = \{x \in \Omega: |x-x_i|_{A_i}^2 - w_i \leq |x-x_j|_{A_j}^2 - w_j\; \forall\, j \in \{1,\ldots,n\}\}, 
    \end{equation}
        where $|\cdot|_{A_i}$ denotes the norm generated by $A_i$, which is defined by $|x|_{A_i} = (x \cdot A_i x)^{1/2}$ for all $x \in \mathbb{R}^d$. The \emph{anisotropic Laguerre tessellation} generated by $(\bfX,\bfw,\mathbf{A})$ is the collection of cells $\{\Lag^{\mathbf{A}}_i(\bfX,\bfw)\}_{i=1}^n$.
    \end{definition}

    The standard Laguerre 
    %cells are
    tessellation is
    recovered by setting $A_i = I_{d}$ for all $i$, where $I_d$ is the $d \times d$ identity matrix. The matrices $A_i$ allow for some control over the aspect ratios of the cells. %grains. 
    If $d=2$ each edge of an anisotropic Laguerre cell is an arc of a conic section.
    %The previous theory can be generalised to this setting. 

We briefly describe how Goal \ref{goal: recovering} can be addressed in the anisotropic setting, namely how to recover an anisotropic Laguerre tessellation given the volumes and centroids of its cells and the anisotropy matrices $\mathbf{A}$. We also state a uniqueness result.

    Define the anisotropic optimal transport problem
     \begin{equation} 
     %\mathcal{T(\bfX,\bfw,\mathbf{A})} 
     \mathcal{T}_\mathbf{A}(\bfX,\bfv)
     = \underset{T \in \mathcal{A}(\bfX,\bfv)}{\operatorname{inf}} \sum_{i=1}^{n}\int_{T^{-1}(\{x_i\})} |x-x_i|^2_{A_i} \, \mathrm{d}x,
    \end{equation}
    where $\mathcal{A}(\bfX,\bfv)$ is the set of admissible transport maps, which was defined in equation \eqref{eq:admissible maps}.
    % , explicitly, 
    % \begin{equation}
    %   \mathcal{A}(\bfX,\bfv)=\{T:\Omega \mapsto \bfX\,: \; \vol(T^{-1}(\{x_i\})) = \tvoli_i \; \forall \; i \in \{1,\ldots,n\}\}. 
    % \end{equation}
    %\db{\sout{Analogously, $^{\mathbf{A}}W_2^{2}(\mathcal{L}^d_{\Omega},\nu(\bfX;\bfv)): = \mathcal{T}^{1/2}(\bfX,\bfw)$ is the \textcolor{red}{CHECK IT IS A METRIC} anisotropic Wasserstein distance.}} 
    As above, it is well-known that the optimal map $T^*:\Omega \to \{x_1,\ldots,x_n\}$ partitions the domain $\Omega$ into anisotropic Laguerre cells, namely, there exists $\bfw = \bfw^{\star}(\bfX;\bfv,\mathbf{A}) \in \mathbb{R}^n$ such that $(T^*)^{-1}(\{ x_i \}) = \Lag^{\mathbf{A}}_i(\bfX,\mathbf{w})$ for all $i$. The vector $\bfw^{\star}(\bfX;\bfv,\mathbf{A})$ maximises the continuously differentiable, concave dual function
    $\mathcal{K}_{\mathbf{A}} : \mathbb{R}^n \to \mathbb{R}$,
    \[
    \mathcal{K}_{\mathbf{A}}(\bfw) = \sum_{i=1}^{n} \int_{\Lag_i^{\mathbf{A}}(\bfX,\bfw)} (|x-x_i|_{A_i}^2 - w_i) \,\mathrm{d}x + \sum_{i=1}^n w_i\tvoli_i,
    \]
    which has gradient given by
    \[
    \frac{\partial \mathcal{K}_{\mathbf{A}}}{\partial w_i}(\bfw) = v_i - \mathrm{vol} (\Lag_i^{\mathbf{A}}(\bfX,\bfw)) \quad \forall \; i \in \{1,\ldots,n\}.
    \]
    In particular, 
    \[
    \mathrm{vol} \big( \Lag_i^{\mathbf{A}}(\bfX,\bfw^{\star}(\bfX;\bfv,\mathbf{A})) \big) = v_i \quad \forall \; i \in \{1,\ldots,n\}.
    \]
    These results essentially follow from \cite[Proposition 37, Theorem 40]{MerigotThibertOT}. (A technical remark for experts in optimal transport theory: anisotropic Laguerre cells intersect in a set of Lebesgue-measure zero and hence the theory from \cite{MerigotThibertOT} applies even though the anisotropic transport cost $c:\Omega \to \{ x_1,\ldots,x_n \}$ given by $c(x,x_i)=|x-x_i|_{A_i}^2$ does not satisfy the assumption of being \emph{twisted} in general.)
We define $\volcell^{\mathbf{A}}_i(\bfX; \bfv)$
to be the $i^{\text{th}}$ anisotropic Laguerre cell in the tessellation with seeds $\bfX$, cell volumes $\bfv$ and anisotropy matrices $\mathbf{A}$, and $\cent_i(\bfX;\bfv,\mathbf{A})$ to be its centroid:
\[
\volcell^{\mathbf{A}}_i(\bfX; \bfv) = \Lag_i(\bfX,\bfw^{\star}(\bfX;\bfv,\mathbf{A})),  
\qquad          
\cent_i(\bfX;\bfv,\mathbf{A})  = \sigma \big( \Lag_i(\bfX,\bfw^{\star}(\bfX;\bfv,\mathbf{A})) \big).  
\]

The following theorem, which generalises Theorem \ref{thm:properties H} \ref{Thrm4.2concave},\ref{Thrm4.2:HisC1}, shows how an anisotropic Laguerre tessellation can be reconstructed from the volumes and centroids of its cells, given the anisotropy matrices $A_i$, 
by maximising a concave function.

\begin{theorem}[Recovering an anisotropic Laguerre tessellation]
\label{thm:aniso}
Given $\vB \in \mathcal{D}_n$ and a tuple of symmetric positive-definite matrices $\mathbf{A} = (A_1,\ldots,A_n) \in (\mathbb{R}^{d \times d})^n$, define the function $H_{\mathbf{A}}:\RR^{nd} \rightarrow \RR$ by
\begin{equation}
   \label{eq:H_A}
H_{\mathbf{A}}(\bfX) = H_{\mathbf{A}}(\bfX;\vB) 
= 
\frac{1}{2} \mathcal{T}_\mathbf{A}(\bfX,\bfv)
- \frac{1}{2}\sum_{i=1}^{n} v_i |x_i|^2_{A_i} + \sum_{i=1}^{n} v_i x_i \cdot (A_i b_i).
\end{equation}
If $\bfX \in \distinct_n$, then
\[
H_{\mathbf{A}}(\bfX) = \frac 12 \sum_{i=1}^n \int_{\volcell^{\mathbf{A}}_i(\bfX; \bfv)} |x-x_i|_{A_i}^2 \, \mathrm{d}x
- \frac{1}{2}\sum_{i=1}^{n} v_i |x_i|^2_{A_i} + \sum_{i=1}^{n} v_i x_i \cdot (A_i b_i).
\]
Then $H_{\mathbf{A}}$ is concave, continuously differentiable on the set of distinct seeds, $H_{\mathbf{A}} \in C^1(\mathbb{D}_n)$, 
and for all $\bfX \in \mathbb{D}_n$ the gradient of $H_{\mathbf{A}}$ is given by
\begin{equation}
\label{eq:grad H_A}    
\frac{\partial H_{\mathbf{A}}}{\partial x_i}(\bfX) = v_i A_i(b_i - c_i(\bfX;\bfv,\mathbf{A})) \quad \forall \; i \in \{1,\ldots,n\}.
\end{equation}
In particular, $\nabla H_{\mathbf{A}}(\bfX) = 0$ if and only if 
$\bfX$ generates an anisotropic Laguerre tessellation with anisotropy matrices $A_i$ and cells of volume $v_i$ and centroids $b_i$.
\end{theorem}

 Note that $H_\mathbf{A}$ does not contain the constant term $-\frac 12 \int_\Omega |x|^2 \, \mathrm{d}x$ that is present in the definition of $H$ (compare equations \eqref{eq:H} and \eqref{eq:H_A}). This is because 
\[
\sum_{i=1}^n \int_{\volcell^{\mathbf{A}}_i(\bfX; \bfv)} |x|_{A_i}^2 \, \mathrm{d}x \ne \int_\Omega |x|^2 \, \mathrm{d}x
\]
in general (unless $A_i=I_d$ for all $i$). Consequently the analogue of Theorem \ref{thm:properties H} \ref{Thrm4.2H:cents} does not hold, namely, the maximum value of $H_\mathbf{A}$ is not zero if the data $\bfv$ and $\bfB$ are the volumes and centroids of an anisotropic Laguerre tessellation.

\begin{proof}[Sketch proof of Theorem \ref{thm:aniso}] We just sketch the proof for brevity and because it is similar to the proof of Theorem \ref{thm:properties H}. The proof that $H_{\mathbf{A}}$ is concave is exactly the same as the proof of Theorem \ref{thm:properties H} \ref{Thrm4.2concave}. We derive the gradient of $H_{\mathbf{A}}$ as follows.
Let $\bfX \in \distinct_n$ and $\bfY \in \mathbb{R}^{nd}$. Then
\begin{align*}
 H_\mathbf{A}(\bfY)
& \ge 
\frac 12 \sum_{i=1}^n \int_{\volcell^{\mathbf{A}}_i(\bfX; \bfv)} |x-y_i|_{A_i}^2 \, \mathrm{d}x
- \frac{1}{2}\sum_{i=1}^{n} v_i |y_i|^2_{A_i} + \sum_{i=1}^{n} v_i y_i \cdot (A_i b_i)
\\
& = 
\frac 12 \sum_{i=1}^n \int_{\volcell^{\mathbf{A}}_i(\bfX; \bfv)} |x-x_i|_{A_i}^2 \, \mathrm{d}x
- \frac{1}{2}\sum_{i=1}^{n} v_i |x_i|^2_{A_i} 
+ \sum_{i=1}^{n} v_i x_i \cdot (A_i b_i)
\\
& \qquad +
\sum_{i=1}^{n} v_i (y_i - x_i) \cdot (A_i b_i)
+
\sum_{i=1}^n \int_{\volcell^{\mathbf{A}}_i(\bfX; \bfv)} (x_i-y_i) \cdot (A_i x) \, \mathrm{d}x 
\\
& =  
H_\mathbf{A}(\bfX)
+ \sum_{i=1}^n (y_i-x_i) \cdot v_iA_i (b_i -   c_i(\bfX;\bfv,\mathbf{A})).
\end{align*}  
In particular, the superdifferential of the concave function $H_{\mathbf{A}}$ at $\bfX$ contains the point $\big( v_iA_i (b_i -   c_i(\bfX;\bfv,\mathbf{A})) \big)_{i=1}^n$. Similarly to the proof of \cite[Theorem 40]{MerigotThibertOT}, it can be shown that this is the only point in the superdifferential and that it depends continuously on $\bfX$. Therefore $H_\mathbf{A}$ is continuously differentiable on $\distinct_n$ with gradient given by \eqref{eq:grad H_A}, as required.

Alternatively, the gradient of $H_\mathbf{A}$ can be derived as follows. Let $\bfX \in \distinct_n$. By using the Kantorovich Duality Theorem from optimal transport theory, we can write
\begin{align*}
H_\mathbf{A}(\bfX) 
& = \max_{\bfw \in \mathbb{R}^n}
\left( 
\frac 12 \int_{\Omega} \min_i (|x-x_i|_{A_i}^2 - w_i) \, \mathrm{d}x + \sum_{i=1}^n v_i w_i
- \frac{1}{2}\sum_{i=1}^{n} v_i |x_i|^2_{A_i} + \sum_{i=1}^{n} v_i x_i \cdot (A_i b_i)
\right)
\\
& =: \max_{\bfw \in \mathbb{R}^n} G_{\mathbf{A}}(\bfX,\bfw).
\end{align*}
Then it can be shown that
\begin{equation}
\label{eq:Taylor}
\frac{\partial H_{\mathbf{A}}}{\partial x_i}(\bfX)
= \frac{\partial G_{\mathbf{A}}}{\partial x_i}(\bfX,\bfw^{\star}(\bfX;\bfv,\mathbf{A})),
\end{equation}
where $\bfw^{\star}(\bfX;\bfv,\mathbf{A}) \in \mathbb{R}^n$ is any maximiser of $G_{\mathbf{A}}(\bfX,\cdot \,)$. Differentiating $G_\mathbf{A}$ gives
\begin{align}
\nonumber
\frac{\partial G_{\mathbf{A}}}{\partial x_i}(\bfX,\bfw)
& = 
\frac 12 \int_{\Omega} \frac{\partial}{\partial x_i} \min_k (|x-x_k|_{A_k}^2 - w_k) \, \mathrm{d}x 
- v_i A_i x_i + v_i A b_i
\\
& = 
\frac 12 \sum_{k=1}^n \int_{\Lag_k^\mathbf{A}(\bfX,\bfw)} \frac{\partial}{\partial x_i} (|x-x_k|_{A_k}^2 - w_k) \, \mathrm{d}x 
- v_i A_i x_i + v_i A b_i
\\
\label{eq:Swift}
& = \int_{\Lag_i^\mathbf{A}(\bfX,\bfw)} A_i(x_i - x) \, \mathrm{d}x 
- \sum_{i=1}^n v_i A_i x_i + \sum_{i=1}^n v_i A b_i
\end{align}
for any $\bfw \in \mathbb{R}^n$.
Combining \eqref{eq:Taylor} and \eqref{eq:Swift} gives the gradient of $H_\mathrm{A}$, as claimed.
\end{proof}

The following result generalises Theorem \ref{thm:uniqueness}.    
    %The $\bfw$ which induces the optimal partition can be found by solving the corresponding dual problem, which we denote $\bfw^{\star}(\bfX;\bfv)$.  
%We also still have uniqueness, for a fixed $\mathbf{A}$.
     \begin{theorem}[Uniqueness of anisotropic compatible diagrams]
     Let $(\bfv,\bfB) \in \mathcal{D}_n$
         and let $\mathbf{A} = (A_1,\ldots,A_n) \in (\mathbb{R}^{d \times d})^n$ be a tuple of symmetric positive-definite matrices.
            Suppose $\bfX, \bfY \in \distinct_{n}$ are such that
            $\{ L^{\mathbf{A}}_i(\bfX;\bfv) \}_{i=1}^n$ and $\{ L^{\mathbf{A}}_i(\bfY;\bfv) \}_{i=1}^n$ are compatible with the data $(\bfv,\bfB)$ in the sense that
            \[
                \cent_i(\bfX;\bfv,\mathbf{A}) = \cent_i(\bfY;\bfv,\mathbf{A}) = b_i 
                \quad \forall \; i \in \{1,\ldots,n\}.
            \]
            Then
            \begin{equation*}
                L^{\mathbf{A}}_i(\bfX;\bfv) = L^{\mathbf{A}}_i(\bfY;\bfv)  \quad \forall \; i \in \{1,\ldots,n\}.
            \end{equation*}
            In other words, an anisotropic Laguerre tessellation is uniquely determined by the volumes and centroids of its cells and its anistropy matrices.
            \end{theorem}
            \begin{proof}
            The proof is almost identical to the proof of Theorem \ref{thm:uniqueness} and so we do not repeat it here. It is simply a matter of replacing Euclidean norms with anisotropic norms.
            \end{proof}

\section{Conclusions and future directions}

In this paper we have given a complete theoretical answer to Goal \ref{goal: recovering}; we proved that a Laguerre tessellation is uniquely determined by the volumes and centroids of its cells and that it can be recovered using convex optimisation (Theorems \ref{thm:uniqueness} and \ref{cor:soln to Goal 1}). 
We performed some preliminary numerical experiments in Section \ref{sec: numerics - diagram exists} but there is room for improvement here. Since we used a gradient-based method to maximise $H$, and $\nabla H$ is only defined where the seeds are distinct, we introduced the extra constraints $c_{ij} \ge 0$, which destroyed the convexity of the constraint set. In future work it would be interesting to use non-smooth convex optimisation (such as a proximal method) to maximise $H$ on the ball and preserve the convexity of the problem. 

Goal \ref{goal: fitting} turned out to be more challenging and, for some data $\vB$, even to be ill-posed.
We start with the positive results.
If the target data $\vB$ satisfies the implicit assumption that the maximiser of $H$ on a ball is achieved at a point with distinct seeds (and under some other generic assumptions), then we found a complete theoretical answer to Goal \ref{goal: fitting}; local minimisers of the fitting error $f$ can be found by solving a convex optimisation problem (Theorem \ref{thm:main}). As for Goal \ref{goal: recovering}, our numerical simulations could be improved to exploit this convexity. Nevertheless, we still obtained good results for the materials science application in Section \ref{sec:EBSD}. 

The implicit assumption on the target data $\vB$, however, is restrictive. In Section \ref{subsec: larger perts} we gave numerical evidence that there exists data $\vB$ such that the maximiser of $H$ on a ball is achieved at a point in the set $\mathbb{R}^{nd} \setminus \distinct_n$ of non-distinct seeds, where $\nabla H$ is not defined. In this case we believe that the infimum of $f=|\nabla H|^2$ is not attained. The non-existence of a minimiser of $f$ on $\distinct_n$ can also be checked analytically for a very simple example with two seeds. To obtain a well-posed optimisation problem for all data $\vB$, in Sections \ref{sec:fitting} and \ref{sec:EBSD} we supplemented \eqref{eq:NLS} with the constraints $c_{ij} \ge 0$. An alternative approach could be to use a different objective function entirely. For example, we could drop the constraint that the volumes of the cells are fitted exactly, and replace the centroid error $f(\bfX)$ by a function of $(\bfX,\bfw)$ measuring a weighted sum of the volume error and the centroid error.

It is an open problem to characterise the data $\vB$ for which \eqref{eq:NLS} attains its infimum. The simulations in Section \ref{subsec: larger perts} suggest that this set could be quite a small neighbourhood of the data $\vB$ for which there exists a compatible diagram. It is also an open problem to find sufficient conditions on $\vB$ for there to exist a compatible diagram; in Section \ref{sec:existence} we only gave necessary conditions. We hope that this paper inspires further work on Goal \ref{goal: fitting}.

\paragraph{Acknowledgements.}
The authors thank Piet Kok, Filippo Santambrogio, Karo Sedighiani and Wil Spanjer for fruitful discussions. The EBSD image in Figure \ref{fig:EBSD} was provided by Tata Steel Europe. DPB acknowledges financial support from the EPSRC grant EP/V00204X/1 Mathematical Theory of Polycrystalline Materials. MP thanks the Centre for Doctoral Training in Mathematical Modelling, Analysis and Computation (MAC-MIGS), 
funded by EPSRC grant 
EP/S023291/1.  

\paragraph{Data Availability Statement.}
The code used in this paper is available online in a GitHub repository:
\url{https://github.com/DPBourne/Laguerre-Polycrystalline-Microstructures}
\cite{GitHub}.

\bibliographystyle{plain}
\bibliography{Bibliography}
        
\end{document}

%% file: preamble.tex
\usepackage{graphicx} % For including graphics (via \includegraphics).
\usepackage{amsmath}  % Improves typographic quality of mathematical output.
\usepackage{amsthm}
\usepackage{amsfonts} % For mathematical fonts.
\usepackage{xcolor}
\usepackage{amssymb}
\usepackage{tikz} % Extends the figure generation capabilities of LaTeX.
\usepackage{hyperref}

\usepackage{bm}
\usepackage{enumitem}

\usepackage{a4wide}
\usepackage[normalem]{ulem}

\newcommand{\bfX}{\mathbf{X}}
\newcommand{\bfw}{\mathbf{w}}
\newcommand{\bfY}{\mathbf{Y}}
\newcommand{\bfZ}{\mathbf{Z}}
\newcommand{\ldomega}{\mathcal{L}^d_{\Omega}}
\newcommand{\bfv}{{\bf{v}}}

\newcommand{\bfwstarX}{{\bfw^*(\bfX;{\bf{v}})}}
\newcommand{\Lag}{\mathrm{Lag}}
\newcommand{\vol}{\mathrm{vol}}

\newcommand{\bfB}{\mathbf{B}}
\newcommand{\cent}{c}
\newcommand{\vB}{(\bfv,\bfB)}
\newcommand{\volcell}{L} % Notation for cell with correct volume
\newcommand{\setcent}{\sigma}
\newcommand{\cellcent}{\xi}
\newcommand{\cellvol}{m}

\newcommand{\tvoli}{v}
\newcommand{\tcenti}{{b}}
\newcommand{\distinct}{\mathbb{D}}

\numberwithin{equation}{section}

\newcommand{\RR}{\mathbb{R}}

\newtheorem{theorem}{Theorem}[section]
\newtheorem{lemma}[theorem]{Lemma}
\newtheorem{proposition}[theorem]{Proposition}

\newtheorem{goal}{Goal}
\theoremstyle{definition}
\newtheorem{definition}[theorem]{Definition}
\newtheorem{remark}[theorem]{Remark}
\newtheorem{example}[theorem]{Example}
\definecolor{brightblue}{rgb}{0,0,1}
\definecolor{darkpastelgreen}{rgb}{0.01, 0.75, 0.24}